\documentclass[a4paper]{amsart}

\usepackage[utf8]{inputenc} 
\usepackage[T1]{fontenc}
\usepackage{tgtermes}
\usepackage{newtxmath}
\usepackage[english]{babel}
\usepackage{graphicx}
\usepackage{listings}
\usepackage{mathtools}
\usepackage{mathrsfs}
\usepackage{color}
\usepackage{geometry}
\usepackage{titlesec}

\newcommand{\RR}{\mathbb{R}}
\newcommand{\CC}{\mathbb{C}}
\newcommand{\NN}{\mathbb{N}}

\newtheorem{theorem}{Theorem}[section]
\newtheorem{proposition}[theorem]{Proposition}
\newtheorem{Lem}[theorem]{Lemma}
\newtheorem{definition}[theorem]{Definition}
\newtheorem{Cor}[theorem]{Corollary}
\newtheorem{Claim}[theorem]{Claim}

\theoremstyle{remark}
\newtheorem{Rq}[theorem]{Remark}

\titleformat{\section}[block]
{\centering\scshape\Large\bfseries}
{\normalfont\textbf{\thesection}.}{0.5em}{}

\titleformat{\subsection}[runin]
{\normalfont\bfseries}
{\thesubsection.}{0.5em}{}[.]
\titlespacing{\subsection}{1pc}{1.5ex plus .1ex minus .2ex}{0.5pc}

\numberwithin{equation}{section}

\title[Non dispersive solutions of the gKdV equations]{Non dispersive solutions of  the generalized KdV equations are typically multi-solitons}

\author{Xavier Friederich}
\address{Institut de Recherche Mathématique Avancée UMR 7501, Université de Strasbourg, Strasbourg, France}
\email{friederich@math.unistra.fr}

\begin{document}

\begin{abstract}
We consider solutions of the generalized Korteweg-de Vries equations (gKdV) which are non dispersive in some sense (in the spirit of \cite{asympstab1}) and which remain close to multi-solitons. We show that these solutions are necessarily pure multi-solitons. For the Korteweg-de Vries equation (KdV) and the modified Korteweg-de Vries equation (mKdV) in particular, we obtain a characterization of multi-solitons and multi-breathers in terms of non-dispersion.
\end{abstract}

\maketitle

\let\thefootnote\relax\footnotetext{2010 \textit{Mathematics Subject Classification:} Primary 35B53, 35Q53; Secondary 35B40, 35B65.}
\let\thefootnote\relax\footnotetext{ \textit{Key words:} generalized KdV equations; solitons; multi-solitons.}

\section{Introduction}

\subsection{Setting of the problem and known results}

We consider the generalized Korteweg-de Vries equations
\begin{gather}\tag{gKdV}
\begin{dcases}
\partial_t u+\partial_x(\partial^2_xu+u^p)=0\\
u(0)=u_0\in H^1(\RR)
\end{dcases}
\end{gather}
where $(t,x)$ are elements of $\RR\times\RR$ and $p>1$ is an integer.\\

Recall that the Cauchy problem for (gKdV) is locally well-posed in $H^1(\RR)$ from a standard result by Kenig, Ponce and Vega \cite{kenig} and that the two following quantities are conserved for each solution $u$ of (gKdV) for all $t$:
\begin{itemize}
\item (the $L^2$-mass) $\displaystyle\int_\RR u^2(t,x)\;dx$
\item (the energy) $\displaystyle\int_\RR\left\{\frac{1}{2} u_x^2-\frac{1}{p+1}u^{p+1}\right\}(t,x)\;dx$.
\end{itemize}

In addition, the set of solutions of (gKdV) is conserved under scaling transformation 
\[
u\longmapsto \left((t,x)\mapsto \lambda^{\frac{2}{p-1}}u\left(\lambda^3t,\lambda x\right)\right),\]
for all $\lambda>0$, and the $\dot{H}^{\sigma(p)}$-norm is invariant under this transformation, where $\sigma(p):=\frac{1}{2}-\frac{2}{p-1}$. Let us recall that the global dynamics of the solutions depends on the sign of $\sigma(p)$. The case $\sigma(p)<0$ that is $1<p<5$, is called $L^2$-subcritical, and all $H^1$-solutions of (gKdV) are then global (in time) and $H^1$-uniformly bounded. If $\sigma(p) =0$, that is $p=5$, we are in $L^2$-critical case and solutions might blow up in finite time \cite{mmcritical,martel2002,mmrI,mmrII,mmrIII}. In the $L^2$-supercritical case, corresponding to  $\sigma(p)>0$ (or $p>5$), much less is known but finite time blow up is expected: existence of $p^*>5$ and of blow-up solutions for all $p\in(5,p^*)$ are proven in \cite{lan}. \\

Moreover it is well-known that (gKdV) admits a family of explicit traveling wave solutions indexed by $\RR^*_+\times\RR$. Let $Q$ be the unique (up to translation) positive solution in $H^1(\RR)$ (known also as \textit{ground state}) to the following stationary elliptic problem associated with (gKdV)
\[ Q''+Q^p=Q, \]
given by the explicit formula
\[ Q(x)=\left(\frac{p+1}{2\mathrm{ch}^2\left(\frac{p-1}{2}x\right)}\right)^{\frac{1}{p-1}}. \]
Then for all $c_0>0$ (velocity parameter) and $x_0\in\RR$ (translation parameter), 
\begin{equation}\label{soliton} 
R_{c_0,x_0}(t,x)=Q_{c_0}(x-c_0t-x_0)\end{equation} 
is a global traveling wave solution of (gKdV) classically named \textit{soliton} solution, where $Q_{c_0}(x)=c_0^{\frac{1}{p-1}}Q(\sqrt{c_0}x).$ It is orbitally stable if and only if $p<5$ ($L^2$-subcritical case) (see Weinstein \cite{weinstein}, Bona, Souganidis and Strauss \cite{bona}, Grillakis, Shatah and Strauss \cite{grillakis}, and Martel and Merle \cite{mm_instabilite}). \\

Solitons are special objects which enjoy very specific properties. Let us recall the following rigidity result, which roughly asserts that non dispersive solutions to (gKdV) which are close to solitons are actually exactly solitons.

\begin{theorem}[Liouville property near a soliton; Martel and Merle \cite{asympstabcrit,asympstab}]\label{liouville}
Let $c_0>0$. There exists $\alpha>0$ such that if $u\in\mathscr{C}(\RR,H^1(\RR))$ is a solution of (gKdV) satisfying, for some $\mathscr{C}^1$ function $y:\RR\to\RR$,
\begin{equation}\label{hypstab}
\text{(closeness to a soliton)}\qquad
\forall\;t\in\RR,\qquad \|u(t,\cdot+y(t))-Q_{c_0}\|_{H^1}\leq\alpha,
\end{equation}
\begin{equation}\label{L2compsoliton}
\text{(non-dispersion)}\qquad\forall\;\varepsilon>0,\:\exists\;R>0,\:\forall\;t\in\RR,\qquad\int_{|x|>R}u^2(t,x+y(t))\;dx\leq\varepsilon, 
\end{equation}
then there exist $c_1>0$, $x_1\in\RR$ such that $$\forall\;t,x\in\RR,\qquad u(t,x)=Q_{c_1}(x-x_1-c_1t).$$
\end{theorem}

This striking result has its own interest of course, but we emphasize that it is also a key ingredient to prove asymptotic stability of (gKdV) solitons (we refer to \cite{asympstabcrit,asympstab0,asympstab}).
We highlight the fact that this result applies in each mass subcritical, critical, and supercritical case by requiring the solution $u$ to remain close to a soliton (up to translation) for all times (\ref{hypstab}). In the $L^2$-subcritical case where solitons are known to be stable, (\ref{hypstab}) can be relaxed to hold only at $t=0$. \\

Finally let us note that solitons play a fundamental role in the study and the understanding of the (gKdV) flow; the important soliton resolution conjecture asserts that any solution with generic initial condition behaves as a sum of solitons plus a radiative-dispersive term as time goes to infinity. In this spirit, built upon solitons, we are interested in other solutions to our problem, namely multi-soliton solutions, defined as follows.

\begin{definition} \label{def:multi_sol}
Let $N\geq 1$ and consider $N$ solitons $R_{c_i,x_i}$ as in \eqref{soliton} with speeds $0<c_1<\cdots<c_N$. A \textit{multi-soliton} \textit{in $+\infty$} (resp. \textit{in $-\infty$}) associated with the $R_{c_i,x_i}$ is an $H^1$-solution $u$ of (gKdV) defined in a neighborhood of $+\infty$ (resp. $-\infty$) and such that \begin{equation}\label{multisol}
\left\|u(t)-\sum_{i=1}^NR_{c_i,x_i}(t)\right\|_{H^1}\rightarrow 0,\qquad \text{as } t\to +\infty \text{ (resp. as } t\to-\infty\text{)}.
\end{equation}
\end{definition}

Multi-solitons are known to exist for all $p>1$; they are even explicit for $p=2$ (KdV) \cite[section 16]{miura} and for $p=3$ (mKdV) \cite[Chapter 5, formula (5.5)]{schuur}. What is more, the classification of the multi-solitons of (gKdV) is complete. Let us gather the main results.

\begin{theorem}[Martel \cite{martel}; Côte, Martel and Merle \cite{cote}; Combet \cite{combetgkdv}]\label{th_multi_connu}
Let $p>1$ be an integer and let $N\ge 1$, $0<c_1^+<\dots<c_N^+$, and $x_1^+,\dots, x_N^+\in\RR$. \\
If $p\le 5$, there exists $T_0\ge 0$ and a unique multi-soliton $u\in\mathscr{C}([T_0,+\infty),H^1(\RR))$ associated with the $R_{c_i^+,x_i^+}$, $i\in\{1,\dots,N\}$.\\
If $p>5$, there exists a one-to-one map $\Phi$ from $\RR^N$ to the set of all $H^1$-solutions of (gKdV) defined in a neighborhood of $+\infty$ such that $u$ is a multi-soliton in $+\infty$ associated with the $R_{c_i^+,x_i^+}$ if and only if there exist $\lambda\in\RR^N$ and $T_0\ge 0$ such that $u_{|[T_0,+\infty)}=\Phi(\lambda)_{|[T_0,+\infty)}$. \\
Moreover, in each case, $u$ belongs to $\mathscr{C}([T_0,+\infty),H^s(\RR))$ for all $s\ge 0$, and there exist $\theta>0$ and positive constants $C_s$ such that for all $s\ge 0$, for all $t\geq T_0$,
\begin{equation}\label{CVexp}
\left\|u(t)-\sum_{i=1}^NR_{c_i^+,x_i^+}(t)\right\|_{H^s}\le C_se^{-\theta t}.
\end{equation}
\end{theorem}

\noindent In the $L^2$-subcritical case (like solitons), sums of decoupled and ordered solitons are stable in $H^1(\RR)$,  even asymptotically stable (Martel, Merle and Tsai \cite{tsai} and Martel and Merle \cite{asympstab}), and  so do multi-solitons.

\subsection{Main results}

Several properties available for solitons have been adapted or even extended to multi-solitons. This article precisely takes this step since it aims at providing an analogue of the rigidity property of Theorem \ref{liouville} in the multi-soliton case.  We consider solutions of (gKdV) that are non dispersive in some sense and uniformly close to the sum of $N$ solitons, and show that they are exact multi-solitons.

\begin{theorem}[Liouville property near a multi-soliton]\label{rigidite}
Let $u$ be a solution of (gKdV) which belongs to $\mathscr{C}([0,+\infty),H^1(\RR))$. Assume the existence of $\rho>0$ such that 
\begin{equation}\label{L2compactnessprop}
\forall\;\varepsilon>0,\:\exists\; R_\varepsilon>0,\:\forall\;t\ge 0,\qquad\int_{x<\rho t-R_\varepsilon}u^2(t,x)\;dx\leq\varepsilon.
\end{equation}
\noindent Let $N\geq 1$ be an integer and consider $N$ positive real numbers $0<c_1<\cdots<c_N$. There exists $\alpha=\alpha(c_1,\dots,c_N,\rho)>0$ such that the following holds: if there exist $N$ functions $x_1,\ldots,x_N:\RR^+\to\RR$ of class $\mathscr{C}^1$ satisfying \begin{equation}\label{Nsolproche}
\forall\;t\geq 0,\qquad \left\|u(t)-\sum_{i=1}^NQ_{c_i}(\cdot-x_i(t))\right\|_{H^1}\leq\alpha,
\end{equation}
and  
\begin{equation}\label{inegxi}
\forall\;t\geq 0,\;\forall\;i\in\{1,\ldots,N-1\},\qquad
x_{i+1}(t)-x_{i}(t)\geq |\ln\alpha|,\\
\end{equation}
\noindent then $u$ is a multi-soliton (in $+\infty$). In other words, there exist $\theta>0$, $0<c_1^+<\cdots<c_N^+$, $x_1^+,\ldots,x_N^+\in\RR$ and positive constants $C_s$ such that for all $t\ge 0$, \eqref{CVexp} is granted.
\end{theorem}

\begin{Rq}
In the $L^2$-subcritical case $1<p<5$, as sum of decoupled solitons are stable, assumptions \eqref{Nsolproche} and \eqref{inegxi} can be relaxed to hold only at time $t=0$.
\end{Rq}

This result is a natural extension of Theorem \ref{liouville} to multi-solitons in $+\infty$, which are the only solutions which are non dispersive in the sense \eqref{L2compactnessprop} (and remain close to a sum of solitons): this is a nice dynamical characterization of multi-solitons among solutions to (gKdV). By contraposition, it means that if a solution $u$ remains in large time sufficiently close to a multi-soliton but is \emph{not} a multi-soliton, then it disperses insofar as \eqref{L2compactnessprop} fails. 

\bigskip

We emphasize that, in contrast with the original statement for one soliton (Theorem \ref{liouville}), the non dispersion assumption \eqref{L2compactnessprop} requires the mass to be located essentially for $x \ge \rho t$ for some small positive speed $\rho>0$ (and almost touches $x=0$); it allows (seemingly) for much more room than in the condition \eqref{L2compsoliton}, which requires that the mass be essentially concentrated in a moving ball of fixed size $R_\varepsilon$. 

Furthermore, the assumptions in Theorem \ref{rigidite} are done only for positive times $t \ge 0$ (and not for all times $t \in \RR$). As it  applies of course to the case of a single soliton, Theorem \ref{rigidite} actually extends and refines Theorem \ref{liouville}.

We must underline that this improvement to focus on the behavior for positive times only is actually very meaningful. Indeed, in view of the above result, a solution which would be non dispersive at times $+\infty$ and $-\infty$ would be a multi-soliton at both ends: but such a behavior is not to be expected, except in the integrable cases of $p=2$ (KdV) and $p=3$ (mKdV) and the Gardner nonlinearity $u^2- \lambda u^3$. To support this, let us refer to the work by Martel and Merle \cite{collision, inelastic,nonexistence} (see also Muñoz \cite{munoz}) on the description of 2-solitons: starting with a 2-soliton solution at $-\infty$ (for the quartic $p=4$ (gKdV)), the collision is almost but not elastic, and there is a non zero defect (which one can quantify), so that it is not a 2-soliton at $+\infty$ (and so, by Theorem \ref{rigidite}, it must be dispersive in the sense of \eqref{L2compactnessprop}).

In principle, the computations in the articles above could extend to $N$-solitons for $N \ge 3$, but it has not been performed yet, and one could still wonder if there is always a defect. If one is willing to assume non dispersion for all time $t\in\RR$, our conclusion is that the solution under consideration is a multi-soliton in $+\infty$ for which all derivatives decay exponentially in space for each fixed values of $t$. More precisely, we have

\begin{Cor}\label{rigidite'}
Let $u\in\mathscr{C}(\RR,H^1(\RR))$ satisfy the assumptions \eqref{Nsolproche}, \eqref{inegxi} of Theorem \ref{rigidite}, and assume (to replace \eqref{L2compactnessprop}) the existence of two constants $0<\sigma<\rho$ such that
\begin{equation}\label{L2compactnessprop'}
\forall\;\varepsilon>0,\:\exists\; R_\varepsilon>0,\:\forall\;t\in\RR,\qquad\int_{\mathscr{B}(\rho t,\sigma|t|+R_\varepsilon)^\complement}u^2(t,x)\;dx\leq\varepsilon.
\end{equation}
Then the conclusion of Theorem \ref{rigidite} holds, and also the following exponential decay property at fixed time, for all $s\in\NN$, and for some possibly larger constant $C_s$:
\begin{equation}\label{expdecay}
\forall\;t\geq 0,\:\forall\;x\in\RR,\qquad \big|\partial^s_xu(t,x)\big|\leq C_s\sum_{i=1}^Ne^{-\theta|x-c_i^+t|}. 
\end{equation} 
\end{Cor}

As it was first observed in \cite{asympstabcrit}, non dispersion (for all times) actually self improves to smoothness and exponential decay in space (outside the center of mass). Of course, this is very relevant for solitons, which exhibit precisely this spatial behavior. But it has yet to be proven that multi-solitons do have spatial exponential decay  \eqref{expdecay} as well; even though it is a very natural conjecture, and that it is known that multi-solitons are smooth. To be able to conclude to \eqref{expdecay}, one has currently to make the assumption \eqref{L2compactnessprop'} (in fact, it would be sufficient to assume that $u$ and $u(-t)$ satisfy \eqref{L2compactnessprop} and to assume in addition that the analog of \eqref{L2compactnessprop} with  $x>(\rho+\sigma)t+R_\varepsilon$ holds for positive times), and for the time being, the above Corollary \ref{rigidite'} is meaningful. 

\begin{Rq}
In the $L^2$-subcritical case $1<p<5$, assumptions \eqref{Nsolproche} and \eqref{inegxi} can be relaxed to hold only a time $t=0$. If they hold for large enough times, positive and negative (or outside of the collision period), the conclusion can be strengthened to $u$ being a multi-soliton at $+\infty$ and $-\infty$, and satisfying \eqref{expdecay} for all $t \in \RR$.
\end{Rq}

In the context of the particular (KdV) equation (corresponding to $p=2$), we claim next a result which gives rise to a simplified characterization of multi-solitons among all $H^1$-solutions.

\begin{theorem}\label{kdv}
Let $p=2$ and $u_0\in \mathscr{S}(\RR)\setminus\{0\}$ be such that the corresponding solution $u$ of (KdV), which is defined globally in time, is non dispersive for positive times, that is, satisfies \eqref{L2compactnessprop}. 
Then $u$ is a multi-soliton (in $+\infty$ and $-\infty$).
\end{theorem}

The proof of this theorem relies on the soliton resolution result for (KdV), set up by Eckhaus and Schuur \cite{eckhaus} and refined in Schuur \cite{schuur}. 

\begin{Rq}\label{rq_hyp_u0}
Requiring that the initial condition $u_0$ belongs to the Schwartz space is not necessary in order to reach the conclusion in Theorem \ref{kdv}. Considering the non dispersion assumption made in Theorem \ref{kdv}, it would be sufficient for example that all derivatives up to order 4 of $u_0$ decay faster than $x^{-11}$ when $x\to +\infty$. Actually, we only need to assume that $u_0$ is smooth enough and decays sufficiently rapidly for $|x|\to +\infty$ for the whole of the inverse scattering method to work, thus for the soliton resolution result for (KdV) to hold \cite{cohen,schuur}. 
However, our goal is not to obtain the most general statement, and for clarity purposes, we will not attempt to optimize the regularity and decay assumptions on $u_0$.
\end{Rq}

Similarly, we can characterize non dispersive solutions of the (mKdV) equation. Recall that, in addition to solitons, (mKdV) admits other particular solutions, known as breathers, which are also important with respect to the soliton resolution conjecture. Breathers do not correspond to a superposition of solitons but are instead periodic in time and can move both in the left and right directions; for all $(\alpha,\beta)\in\RR^*_+\times\RR^*_+$, and for all $x_1,x_2\in\RR$, the breather $B_{\alpha,\beta,x_1,x_2}$ with envelope velocity $\gamma:=\beta^2-3\alpha^2$, phase velocity $\delta:=3\beta^2-\alpha^2$, and translation parameters $x_1,x_2\in\RR$ takes the following expression:
\begin{equation}
B_{\alpha,\beta,x_1,x_2}(t,x):=2\sqrt{2}\partial_x\left[\arctan\left(\frac{\beta}{\alpha}\frac{\sin(\alpha(x-\delta t-x_1))}{\cosh(\beta(x-\gamma t-x_2))}\right)\right].
\end{equation}
We refer to Alejo and Muñoz \cite{alejo} for the introduction and the study of stability in $H^2$ of these solutions. Note that the decomposition result in terms of solitons and breathers available for (mKdV) solutions and stated in \cite[Chapter 5, Theorem 5.1]{schuur} and more recently in \cite[Theorem 1.10]{chen} holds under the assumption that the initial data $u_0$ is \emph{generic} in the following sense: the set of all $\xi\in\CC$ such that the classical Jost solutions $\psi_l(\xi)$ and $\psi_r(\xi)$ to the Zakharov-Shabat system
$$\dbinom{\psi_1}{\psi_2}'=\left(\begin{array}{cc}
-i\xi&u_0\\
-u_0&i\xi
\end{array}\right)\dbinom{\psi_1}{\psi_2}$$ are $\RR$-linearly dependent is finite and consists in the \emph{scattering data}
\begin{equation}\label{scattering_set}
\{i\sqrt{c_1},\dots,i\sqrt{c_{N_1}},\alpha_1+i\beta_1,\dots,\alpha_{N_2}+i\beta_{N_2}\},
\end{equation} with $N_1,N_2\in\NN$, $c_1,\dots,c_{N_1},\alpha_1,\dots,\alpha_{N_2},\beta_1,\dots,\beta_{N_2}\in\RR^*_{+}$ such that
\begin{equation}\label{cond_1}
c_1<\dots<c_N,\qquad\beta_1^2-3\alpha_1^2<\dots<\beta_{N_2}^2-3\alpha_{N_2}^2,
\end{equation}
and \begin{equation}\label{cond_2}
\forall\;(i,j)\in\{1,\dots,N_1\}\times\{1,\dots,N_2\},\quad c_i\neq \beta_j^2-3\alpha_j^2.
\end{equation}
We refer to Schuur \cite[Chapter 4]{schuur}, Chen and Liu \cite[Paragraph 1.2]{chen}, and the references therein for more details concerning genericity. 

Our result on non dispersive solutions of (mKdV) writes as follows.

\begin{theorem}\label{mkdv}
Let $p=3$ and $u_0\in \mathscr{S}(\RR)\setminus\{0\}$ be generic (in the above sense) with scattering data \eqref{scattering_set}, and such that the corresponding global solution $u$ of (mKdV) is non dispersive for positive times (that is, satisfies \eqref{L2compactnessprop}). 

Then $u$ is a multi-breather with positive speeds in $+\infty$: we have $N_1+N_2\ge 1$ and for all $j=1,\dots,N_2$, $$\beta_j^2-3\alpha_j^2>0,$$ and there exist $\gamma>0$, positive constants $C_s$, signs $\epsilon_i=\pm 1$, and real parameters $x_{0,i}, x_{1,j}, x_{2,j}$ such that for all $s\ge 0$, $u$ belongs to $\mathscr{C}([0,+\infty),H^s(\RR))$ and
\[ \forall\;t\ge 0,\qquad \left\|u(t)-\sum_{i=1}^{N_1}\epsilon_iR_{2c_i,x_{0,i}}(t)-\sum_{j=1}^{N_2}B_{\sqrt{2}\alpha_j,\sqrt{2}\beta_j,x_{1,j},x_{2,j}}(t)\right\|_{H^s}\leq C_se^{-\gamma t}.\]
\end{theorem} 

The proof is done by adapting that of Theorem \ref{kdv} by writing the soliton/breather resolution for $p=3$, and then using smoothness and uniqueness of multi-breathers and the estimates in higher order Sobolev spaces proved by Semenov \cite{semenov}.

\begin{Rq}
Let us notice that Remark \ref{rq_hyp_u0} applies also in the context of Theorem \ref{mkdv}.
\end{Rq}

\subsection{Comments}

The proofs of Theorem \ref{rigidite} and Corollary \ref{rigidite'} are in the spirit of the original Liouville result by Martel and Merle \cite{asympstabcrit}, and also the work of Laurent and Martel \cite{laurent} on smoothness and decay of non dispersive solutions. An important ingredient is the observation that a crucial monotonicity formula holds under a much relaxed non dispersion assumption than previously made, see Proposition \ref{pointwise}, which has its own interest. Also we underline a subtle but key difference in the strategy of the proof: we crucially rely at some point on the asymptotic stability of multi-solitons in the energy space (from \cite{asympstab1}, stated in Theorem \ref{asympto}). But let us recall this result itself is a consequence of the rigidity result for one soliton stated in Theorem \ref{liouville}: in some sense, the roles are reversed here. 

In fact, our proofs use and combine several previous results on the (gKdV) flow around solitons and multi-solitons. This sheds a new light on many results established so far and which have their own interest, which are here linked together to yield new statements. It seems to us that this phenomenon is an interesting point of this paper. \\

Our results lead to several open questions. We already mentioned above the first one, but repeat it here: we conjecture that, for each time where defined, multi-solitons for (gKdV) have pointwise exponential decay (along with their derivatives); this is only known in the integrable case, where explicit formulas are known. As second question is whether similar rigidity results (as well as asymptotic stability properties) hold for other dispersive models. The Liouville theorem for solitons holds for the Zakharov-Kuznetsov equation in 2D for example, it would be nice to know if an analog for multi-solitons holds as well. A very natural context is that of the non-linear Schrödinger equations, for which the understanding of non-dispersive solutions remains mostly open.  \\

This article is organized as follows. After the introduction, we present in section \ref{sectionexpdec} a general property of exponential decay satisfied by non-dispersive solutions which is an important new observation and interesting in itself. The third section is then devoted to the proof of Theorem \ref{rigidite} and Corollary \ref{rigidite'}. In the forth section, we consider the integrable case and sketch the proofs of Theorem \ref{kdv} and \ref{mkdv}.

\subsection{Acknowledgments}

The author would like to thank his supervisor Raphaël Côte for suggesting the idea of this work and for fruitful discussions.

\section{Smoothness and exponential decay for non dispersive solutions}\label{sectionexpdec}

The goal of this section is to show the following propositions which extend Laurent and Martel \cite[Theorem 1]{laurent}.  

\begin{proposition}\label{pointwise} 
Let $J$ be a neighborhood of $+\infty$ and $u\in\mathscr{C}(J,H^1(\RR))$  be a solution of (gKdV) which belongs also to $L^\infty(J,H^1(\RR))$. Suppose that there exists $\rho>0$ such that 
\begin{equation}\label{L2comp}
\forall\;\varepsilon>0,\:\exists\; R_\varepsilon>0,\:\forall\; t\in J,\qquad\int_{x<\rho t-R_\varepsilon}u^2(t,x)\;dx\leq\varepsilon.
\end{equation}
Then $u$ belongs to $\mathscr{C}^\infty(J\times\RR)$ and there exists $\kappa>0$ such that for all $k\in\NN$, there exists $K_k>0$ such that \begin{equation}\label{leftpoint}
\forall\;t\in J,\: \forall\;x< \rho t,\qquad \big|\partial^k_xu(t,x)\big|\leq K_ke^{-\frac{\kappa}{2}\left|x-\rho t\right|}.
\end{equation}
\end{proposition}

We state next another generalized version, which is useful in the proof of Corollary \ref{rigidite'}.

\begin{proposition}\label{pointwise2} 
Let $J$ be a neighborhood of $+\infty$ and $u\in\mathscr{C}(J,H^1(\RR))$  be a solution of (gKdV) which belongs also to $L^\infty(J,H^1(\RR))$. Suppose that there exist $\beta,\delta>0$ and two $\mathscr{C}^1$ functions $a,b:J\rightarrow\RR$ such that 
\begin{equation}\label{inegder}
\forall\;t\in J,\qquad \delta\leq a'(t)\leq b'(t)\leq \beta,
\end{equation}
and \begin{equation}\label{L2comp'}
\forall\;\varepsilon>0,\:\exists\; R_\varepsilon>0,\:\forall\; t\in J,\qquad\int_{x<m(t)-R_\varepsilon}u^2(t,x)\;dx\leq\varepsilon,
\end{equation}
where $m(t):=\min\{a(t),b(t)\}$.\\ 
Then $u$ belongs to $\mathscr{C}^\infty(J\times\RR)$ and there exists $\kappa>0$ such that for all $k\in\NN$, there exists $K_k>0$ such that \begin{equation}\label{leftpoint2}
\forall\;t\in J,\: \forall\;x<m(t),\qquad \big|\partial^k_xu(t,x)\big|\leq K_ke^{-\frac{\kappa}{2}\left|x-m(t)\right|}.
\end{equation}
\end{proposition}

\begin{Rq}\label{est_left}
It is to be noticed that, if $J$ in Proposition \ref{pointwise2} is replaced by a neighborhood $J'$ of $-\infty$, then we conclude with an estimate at the right of $M(t):=\max\{a(t),b(t)\}$, or more precisely with the existence of $\kappa>0$ such that for all $k\in\NN$, there exists $K_k>0$ such that
\begin{equation}\label{rightpoint}
\forall\;t\in J',\:\forall\;x>M(t),\qquad \big|\partial^k_xu(t,x)\big|\leq K_ke^{-\frac{\kappa}{2}\left(x-M(t)\right)}.
\end{equation} 
This is justified by the following symmetry property for (gKdV) and the assumption in Propositions \ref{pointwise2}. Denoting $\hat{u}(t,x):=u(-t,-x)$, $\hat{a}(t):=-a(-t)$, $\hat{b}(t):=-b(-t)$, $\hat{m}(t):=-m(-t)$, and $\hat{M}(t):=-M(-t)$, we observe that $u$ satisfies the assumptions of Proposition \ref{pointwise} on a neighborhood $J'$ of $-\infty$ if and only if $\hat{u}$ satifies the same assumptions on $-J'$ (which is a neighborhood of $+\infty$) with $a$, $b$, $m$, and $M$ replaced respectively by $\hat{a}$, $\hat{b}$, $\hat{M}$, and $\hat{m}$ in Proposition \ref{pointwise2}. Thus once we have proved (\ref{leftpoint2}) as stated in Proposition \ref{pointwise2}, we have immediately the pointwise estimate on $\hat{u}(t)$ at the left of $\hat{m}(t)$ for $t\in -J'$, which precisely provides (\ref{rightpoint}), that is the desired pointwise estimate on $u(t)$ at the right of $M(t)$ for $t\in J'$. \\
Obviously, Propositions \ref{pointwise} and \ref{pointwise2} apply in particular with $J=\RR$, in which case both estimates \eqref{leftpoint2} and \eqref{rightpoint} hold.
\end{Rq}

Now, proceeding essentially as Laurent and Martel \cite[Theorem 1]{laurent}, we derive the proof of Proposition \ref{pointwise2}.

\begin{proof}

\noindent \textit{Step 1: Estimates to be established} \\

By the classical Sobolev embedding  $H^1\left(-\infty,m(t)\right)\hookrightarrow L^\infty\left(-\infty,m(t)\right)$, it suffices to see that
\begin{equation}\label{estH1}
\exists\; K>0,\:\forall\;t\in J,\qquad \int_{x<m(t)}\big(u^2(t,x)+u_x^2(t,x)\big)e^{\kappa(m(t)-x)}\;dx\leq K
\end{equation}
holds to have the desired conclusion, that is \eqref{leftpoint2}, for $k=0$ and for almost every $x<m(t)$. Using that $u(t)$ is continuous on $\RR$ by $H^1(\RR)\hookrightarrow \mathscr{C}(\RR)$, we deduce that \eqref{leftpoint2} is true for $k=0$. \\
Similarly, to reach the whole conclusion, we have to show that for each $k\in\NN$, there exists $\tilde{K}_k>0$ such that
\begin{equation}\label{estHk}
\forall\;t\in J,\qquad \int_{x<m(t)}\big(\partial^k_{x}u(t,x)\big)^2e^{\kappa(m(t)-x)}\;dx\leq \tilde{K}_k.
\end{equation}

In order to prove \eqref{estHk}, it is convenient to introduce a well-chosen $\mathscr{C}^1$ function defined on $J$, denoted by $\tilde{m}$, which replaces somehow $m$ in the case where $m$ is not already $\mathscr{C}^1$. By this means, we get around the difficulty of a possible point where $m$ is not differentiable. This is the purpose of the following:

\begin{Claim}\label{function_tildem}
There exists $\tilde{m}:J\to\RR$ of class $\mathscr{C}^1$ such that for all $t\in J$, $m(t)\leq\tilde{m}(t)\leq m(t)+1$, and $\tilde{m}'(t)\geq \delta$.
\end{Claim}

\begin{proof}[Proof of Claim \ref{function_tildem}]
Define $\tilde{m}$ by $$\forall\;t\in J,\quad \tilde{m}(t):=1+\frac{a(t)+b(t)}{2}-\sqrt{1+\left(\frac{b(t)-a(t)}{2}\right)^2}.$$ Then $\tilde{m}$ is $\mathscr{C}^1$ on $J$ and given that $\min\{a,b\}=\frac{a+b}{2}-\frac{|a-b|}{2}$, one can check that $\tilde{m}$ satisfies $m(t)\leq\tilde{m}(t)\leq m(t)+1$ by means of the well-known inequality $$\sqrt{x+y}\leq \sqrt{x}+\sqrt{y},\qquad \text{valid for all } x,y\geq 0.$$
Moreover, by a straightforward computation, we have $$\tilde{m}'(t)\geq \frac{a'(t)+b'(t)}{2}-\frac{b'(t)-a'(t)}{2}\geq a'(t).$$ 
Consequently, Claim \ref{function_tildem} is proved.
\end{proof}

Now, we consider $\tilde{m}$ as in the previous claim. Judging by the fact that for all $t\in J$, $m(t)\leq \tilde{m}(t)$, we can write for all $k\in\NN$
\begin{equation}
\int_{x<m(t)}\left(\partial_x^ku(t,x)\right)^2e^{\kappa(m(t)-x)}\;dx\leq \int_{x<\tilde{m}(t)}\left(\partial_x^ku(t,x)\right)^2e^{\kappa(\tilde{m}(t)-x)}\;dx.
\end{equation}
Thus, to achieve our goal (\ref{estHk}), it suffices to show the existence of $C_k>0$ such that 
\begin{equation}\label{ineqL2bis}
\forall\;t\in\RR,\quad\int_{x<\tilde{m}(t)}\left(\partial_x^ku(t,x)\right)^2e^{\kappa(\tilde{m}(t)-x)}\;dx\leq C_k.
\end{equation}

\noindent\textit{Step 2: Proof of \eqref{ineqL2bis} for $k=0$} \\ 

We will obtain (\ref{ineqL2bis}) by a strong monotonicity property which is the purpose of Lemma \ref{monot} and Lemma \ref{limite} below. \\

\noindent Let us introduce, for some $\kappa>0$ to be determined later, the function $\varphi$ defined by
\[ \varphi(x)=\frac{1}{2}-\frac{1}{\pi}\arctan (e^{\kappa x}). \] 
It satisfies the following properties
\begin{align}\label{prop1}
\exists\;\lambda_0>0,\:\forall\;x\in\RR,\qquad \lambda_0e^{-\kappa|x|} & <-\varphi '(x)<\frac{1}{\lambda_0}e^{-\kappa|x|}, \\
\label{prop2}
\forall\;x\in\RR,\qquad |\varphi^{(3)}(x)| & \leq-\kappa^2\varphi '(x). \\
\label{prop3}
\exists\;\lambda_1>0,\:\forall\;x\geq 0,\qquad\lambda_1e^{-\kappa x} & \leq\varphi(x).
\end{align}

\noindent Moreover, let us observe that 
\begin{equation}\label{egalite_integrale_tildeM}
\int_{x<\tilde{m}(t)}u^2(t,x)e^{\kappa(\tilde{m}(t)-x)}\;dx=\displaystyle\int_{x<0}u^2\left(t,x+\tilde{m}(t)\right)e^{-\kappa x}\;dx,
\end{equation}
and that, for all $x_0<0$,
\begin{equation}\label{ineq1}
\begin{aligned}
\int_{x_0\leq x< 0}u^2\left(t,x+\tilde{m}(t)\right)e^{-\kappa x}\;dx&\leq e^{-\kappa x_0}\int_{x\geq x_0}u^2\left(t,x+\tilde{m}(t)\right)e^{-\kappa (x-x_0)}\;dx\\
&\leq\displaystyle\frac{1}{\lambda_1} e^{-\kappa x_0}\int_{x\geq x_0}u^2\left(t,x+\tilde{m}(t)\right)\varphi(x-x_0)\;dx\\
&\leq\displaystyle\frac{1}{\lambda_1} e^{-\kappa x_0}\int_{\RR}u^2\left(t,x+\tilde{m}(t)\right)\varphi(x-x_0)\;dx.
\end{aligned}
\end{equation}

\noindent By Claim \ref{function_tildem}, for all $t\in J$, $\tilde{m}'(t)\geq\delta$. Therefore there exists $\eta>0$ and an increasing affine function $f:\RR\rightarrow\RR$ such that \begin{equation}\label{defeta}
\forall\;t\in J,\quad -f'(t)+\tilde{m}'(t)\geq \eta.
\end{equation}

\noindent Now, for fixed $t_0\in J$ and $x_0$ in $\RR$, consider 
$$\begin{array}{cccl}
I_{(t_0,x_0)}:&\RR&\rightarrow&\RR^+\\
&t&\mapsto&\displaystyle\int_\RR u^2(t,x+\tilde{m}(t))\varphi\big(x-x_0+f(t)-f(t_0)\big)\;dx.
\end{array}$$

\noindent We have \begin{equation}
\forall\;t\in\RR,\qquad I_{(t_0,x_0)}(t)=\displaystyle\int_\RR u^2(t,x)\varphi\big(x-x_0+f(t)-f(t_0)-\tilde{m}(t)\big)\;dx,
\end{equation} 
so that by derivation with respect to $t$, we obtain 
\begin{equation}\label{derivee}
\begin{aligned}
\displaystyle\frac{dI_{(t_0,x_0)}}{dt}(t)=&\displaystyle-3\int_{\RR}u_x^2(t,x)\varphi '(\tilde{x})\;dx-\big(-f'(t)+\tilde{m}'(t)\big)\int_\RR u^2(t,x)\varphi '(\tilde{x})\;dx\\
&\displaystyle+\int_\RR u^2(t,x)\varphi^{(3)}(\tilde{x})\;dx+\frac{2p}{p+1}\int_\RR u^{p+1}(t,x)\varphi '(\tilde{x})\;dx,
\end{aligned}
\end{equation}
where $\tilde{x}:=x-x_0+f(t)-f(t_0)-\tilde{m}(t)$.

\noindent Set $\kappa:=\sqrt{\frac{\eta}{2}}$. We claim then

\begin{Lem}\label{monot} There exists $C_0>0$ such that for all $x_0\in \RR$, and for all $t_0,t\in J$, 
\begin{equation}
\frac{dI_{(t_0,x_0)}}{dt}(t)\geq -C_0e^{-\kappa\big(-x_0+f(t)-f(t_0)\big)}.
\end{equation}
\end{Lem}

\begin{proof}[Proof of Lemma \ref{monot}] 
Due to the choice of $\kappa$ and property (\ref{prop2}) of $\varphi$, we have 
\begin{equation}\label{int0}
\left|\int_{\RR}u^2(t,x)\varphi^{(3)}(\tilde{x})\;dx\right|\leq - \frac{\eta}{2}\int_\RR u^2(t,x)\varphi '(\tilde{x})\;dx. 
\end{equation}

\noindent Furthermore we control the non-linear part by considering, for $R>0$, $$\displaystyle I_1(t):=\int_{|\tilde{x}|>-x_0-R+f(t)-f(t_0)} u^{p+1}(t,x)\varphi '(\tilde{x})\;dx$$ and $$\displaystyle I_2(t):=\int_{\RR} u^{p+1}(t,x)\varphi '(\tilde{x})\;dx-I_1(t).$$
On the one hand, we have due to (\ref{prop1}) 
\begin{equation}\label{int1}
\begin{aligned}
\displaystyle\big|I_1(t)\big|&\leq \frac{1}{\lambda_0}e^{-\kappa\big(-x_0-R+f(t)-f(t_0)\big)}\left(\int_\RR|u|^{p+1}(t,x)\;dx\right)\notag \leq Ce^{-\kappa\big(-x_0-R+f(t)-f(t_0)\big)},
\end{aligned}
\end{equation}
where we have used the Sobolev embedding $H^1(\RR)\hookrightarrow L^{p+1}(\RR)$ and the fact that $u$ belongs to $L^\infty(J,H^1(\RR))$. Note that $C>0$ is independent of $x_0$, $t_0$, and $t$. \\
On the other, we observe that if $|\tilde{x}|\leq -x_0-R+f(t)-f(t_0)$, then 
$x\leq\tilde{m}(t)-R$ in particular, and therefore by Claim \ref{function_tildem} we have also $x\le m(t)-R+1.$
Thus, it follows
\begin{equation}
\begin{aligned}
\displaystyle\big|I_2(t)\big|&\leq\|u(t)\|_{L^\infty\big(x\le m(t)-R+1\big)}^{p-1}\displaystyle\int_{x\le m(t)-R+1}u^2(t,x)|\varphi '(x)|\;dx\\
&\leq\displaystyle\sqrt{2}\|u(t)\|_{L^2\big(x\le m(t)-R+1\big)}^{\frac{p-1}{2}}\|u_x(t)\|_{L^2\big(x\le m(t)-R+1\big)}^{\frac{p-1}{2}} \displaystyle\int_{\RR}u^2(t,x)|\varphi '(x)|\;dx\\
&\leq\displaystyle\sqrt{2}\|u(t)\|_{L^2\big(x\le m(t)-R+1\big)}^{\frac{p-1}{2}}\sup_{t\in\RR}\|u(t)\|_{H^1}^{\frac{p-1}{2}} \displaystyle\int_{\RR}u^2(t,x)|\varphi '(x)|\;dx.
\end{aligned}
\end{equation}  
By the non-dispersion assumption (\ref{L2comp}), we can choose $R>1$ such that 
$$\sqrt{2}\|u(t)\|_{L^2\big(x\le m(t)-R+1\big)}^{\frac{p-1}{2}}\sup_{t\in\RR}\|u(t)\|_{H^1}^{\frac{p-1}{2}}\leq \frac{p+1}{4p}\eta.$$
Taking into account (\ref{int1}), this leads eventually to the following estimate
\begin{equation}\label{int}
\frac{2p}{p+1}\displaystyle\left|\int_\RR u^{p+1}(t,x)\varphi '(\tilde{x})\;dx\right|\leq -\frac{\eta}{2}\int_{\RR}u^2(t,x)\varphi '(x)\;dx+C_0e^{-\kappa\big(-x_0-R+f(t)-f(t_0)\big)},
\end{equation}
where $C_0:=\frac{2p}{p+1}C$ is independent of $x_0$, $t_0$, and $t$. Gathering (\ref{defeta}), (\ref{int0}), and (\ref{int}) in (\ref{derivee}) we deduce finally
$$\frac{dI_{(t_0,x_0)}}{dt}(t)\geq -3\int_{\RR}u_x^2(t,x)\varphi '(\tilde{x})\;dx-C_0e^{-\kappa\big(-x_0-R+f(t)-f(t_0)\big)}.$$ 

\noindent Thus Lemma \ref{monot} is established.
\end{proof}

As a consequence of the preceding lemma,
\begin{equation}\label{monot1}
\exists\;C_1>0,\:\forall\;x_0\in\RR,\:\forall\;t\geq t_0,\qquad I_{(t_0,x_0)}(t_0)\leq I_{(t_0,x_0)}(t)+C_1e^{\kappa x_0},
\end{equation}
with $C_1$ independent of the parameters $x_0$ and $t_0$. Next, we claim the following:

\begin{Lem}\label{limite}
For fixed $x_0\in\RR$ and $t_0\in J$, $I_{(t_0,x_0)}(t)\rightarrow 0$ as $t\rightarrow+\infty$.
\end{Lem}

\begin{proof}
To show this lemma, we just repeat the arguments given by Laurent and Martel \cite[paragraph 2.1, Step 2]{laurent}.
Let $\varepsilon$ be a positive real number.
By Claim \ref{function_tildem} and by (\ref{L2comp'}), there exists $\tilde{R}>0$ such that 
\[ \int_{x<\tilde{m}(t)-\tilde{R}}u^2(t,x)\;dx\leq\frac{\varepsilon}{2}.\] 
Since $0\leq \varphi\leq 1$, this enables us to see that 
\begin{align}
\displaystyle\int_{x<-\tilde{R}}u^2\left(t,x+\tilde{m}(t)\right)\varphi\big(x-x_0+f(t)-f(t_0)\big)\;dx&\leq\displaystyle\int_{x<\tilde{m}(t)-\tilde{R}}u^2(t,x)\;dx \leq\frac{\varepsilon}{2}.
\end{align}

\noindent Now, recall that $\varphi$ is decreasing so that 
\begin{align}
\MoveEqLeft[4]
\int_{x\geq-\tilde{R}}u^2(t,x+\tilde{m}(t))\varphi\big(x-x_0+f(t)-f(t_0)\big)\;dx \notag\\
&\leq\displaystyle\varphi\big(-\tilde{R}-x_0+f(t)-f(t_0)\big)\|u(t)\|_{L^2}^2\notag\\
&\leq\displaystyle \overline{C}\varphi\big(-\tilde{R}-x_0+f(t)-f(t_0)\big),
\end{align} with $\overline{C}=\|u(t)\|_{L^2}^2$ for all $t\in J$. Moreover, since $f(t)\to +\infty$ as $t\to +\infty$ and $\varphi(x)\to 0$ as $x\to +\infty$, there exists $T\in\RR$ such that for all $t\geq T$, 
$$\overline{C}\varphi\big(-\tilde{R}-x_0+f(t)-f(t_0)\big)\leq\frac{\varepsilon}{2}.$$
Then, for all $t\geq T$,
$$I_{(t_0,x_0)}(t)\leq\frac{\epsilon}{2}+\frac{\varepsilon}{2}=\varepsilon.$$
Hence, we have finished proving Lemma \ref{limite}.
\end{proof}

Due to (\ref{monot1}) and Lemma \ref{limite}, we obtain
\begin{equation}
\forall\;t_0\in J,\;\forall\;x_0\in\RR,\qquad I_{(t_0,x_0)}(t_0)\leq C_1e^{\kappa x_0}.
\end{equation}
Thus, (\ref{ineq1}) leads to: for all $t\in J$,
$$\int_{x_0\leq x<0}u^2(t,x+\tilde{m}(t))e^{-\kappa x}\;dx\leq \frac{C_1}{\lambda_1}.$$
Thus letting $x_0$ tend to $-\infty$, we deduce from \eqref{egalite_integrale_tildeM} that
$$\displaystyle\int_{x<\tilde{m}(t)}u^2(t,x)e^{\kappa(\tilde{m}(t)-x)}\;dx\leq \frac{C_1}{\lambda_1}.$$

\noindent\textit{Step 3: Proof of (\ref{estHk}) for $k\ge 1$}\\

Starting from the fact that for all $t\geq t_0$ and for all $x_0<0$
$$I_{(t_0,x_0)}(t_0)-I_{(t_0,x_0)}(t)\leq \frac{C_1}{\lambda_1}e^{\kappa x_0}+3\int_{t_0}^{t}\int_\RR u_x^2(s,x+\tilde{m}(s))\varphi '(x-x_0+f(s)-f(t_0))\;dx\;ds$$
and arguing like Laurent and Martel \cite[paragraph 2.1 Step 3]{laurent}, one can show 
$$\int_{t_0}^{t_0+1}\int_\RR u_x^2\left(s,x+\tilde{m}(s)\right)e^{\kappa x}\;dx\;ds
\leq C,$$ where $C$ is independent of $x_0$ and $t_0$. \\

Now, one proves by induction on $k\in\NN$ the existence of $C_k\ge 0$ such that for all $t\in\RR$, $\|u(t)\|_{H^k}\le C_k$ and
\begin{equation}\label{est_induction}
\int_{\RR}\left(\partial_x^ku\right)^2\left(t,x+\tilde{m}(t)\right)e^{\kappa x}\;dx+\int_{t}^{t+1}\int_{\RR}\left(\partial_x^ku\right)^2\left(s,x+\tilde{m}(s)\right)e^{\kappa x}\;dx\;ds\le C_k.
\end{equation}
In particular, estimates \eqref{estHk} are then performed. Moreover, we deduce from the equation satisfied by $u$ that the partial derivatives with respect to $x$ and $t$ of all order exist and are continuous, thus $u\in\mathscr{C}^\infty(J\times\RR)$. 

For simplification purposes, we will not explicit the proof of \eqref{est_induction} and refer instead to \cite[paragraph 2.3 and paragraph 2.2 Step 2]{laurent}; the induction argument works since we assume $b'(t)\leq \beta$ in (\ref{inegder}), which implies that $\tilde{m}$ is bounded on $J$. 
\end{proof}

\section{Proof of Theorem {\normalfont\textbf{\ref{rigidite}}} and Corollary \normalfont\textbf{\ref{rigidite'}}}\label{proof}

We split the proof into four steps. The three first steps are common to both theorems and are valid under the hypotheses of Theorem \ref{rigidite}, whereas the last one is specific to the proof of Corollary \ref{rigidite'} where exponential decay properties are established and for which the stronger non dispersion assumption \eqref{L2compactnessprop'} is required. \\
Consider $u$ which satisfies the assumptions of Theorem \ref{rigidite}.\\ 

\noindent\textit{Step 1: Asymptotic stability in the energy space} \\

The following asymptotic stability result in the energy space is to be considered as a crucial tool for the proof. 

\begin{theorem}[Martel, Merle and Tsai \cite{tsai}; Martel and Merle \cite{asympstab}]\label{asympto}
Fix $0<c_1^0<\cdots<c_N^0$. For all $\beta>0$, there exist $L_0>0$ and $\alpha_0=\alpha_0(\beta)>0$ such that if $u\in\mathscr{C}([0,+\infty),H^1(\RR))$ is a solution of (gKdV) satisfying
\begin{equation}
\forall\;t\geq 0,\qquad \inf_{\begin{array}{c} r_i\in\RR\\ r_{i+1}-r_{i}>L_0\end{array}}\left\|u(t)-\sum_{i=1}^NQ_{c_i^0}(\cdot-r_i)\right\|_{H^1}<C\big(\alpha_0+e^{-\gamma t}\big)
\end{equation} for some positive constants $C$ and $\gamma$, then the following holds.
\begin{enumerate}
\item \textit{(Asymptotic stability in the energy space)} There exist $\mathscr{C}^1$ functions $t\mapsto c_i(t)\in\RR^*_+$, $t\mapsto \rho_i(t)\in\RR$ for $i\in\{1,\ldots,N\}$ such that
\begin{equation}
\lim_{t\to+\infty}\left\|u(t)-\sum_{i=1}^NQ_{c_i(t)}(\cdot-\rho_i(t))\right\|_{H^1\big(x>\beta t\big)}=0.
\end{equation}
\item \textit{(Convergence of the scaling parameter)} There exists $c_{i}^+\in\RR^*_+$ such that $\displaystyle\lim_{t\to+\infty}c_i(t)=c_{i}^+$.
\end{enumerate}
\end{theorem}

Set $\delta:=\frac{1}{2}\min\left\{c_1,\min_{i\in\{1,\dots,N-1\}}\{c_{i+1}-c_i\}\right\}$.\\
\noindent Using Theorem \ref{asympto} and adapting the classical modulation argument set up in the proof by Martel and Merle \cite[section 5]{asympstab}, we have the following. For all $i\in\{1,\ldots,N\}$, there exist $c_i:[0,+\infty)\rightarrow\RR^*_+$, $\rho_i:[0,+\infty)\rightarrow\RR$ of class $\mathscr{C}^1$ such that, defining
\begin{equation}
\epsilon:(t,x)\mapsto u(t,x)-\sum_{i=1}^NQ_{c_i(t)}(x-\rho_i(t)),
\end{equation}
and for $\alpha$ small enough in Theorems \ref{rigidite} and \ref{rigidite'}, we have 
\begin{enumerate}
\item the conclusion of Theorem \ref{asympto}, that is 
\begin{equation}\label{right}
\lim_{t\to+\infty}\|\epsilon(t)\|_{H^1\big(x>\frac{c_1}{A}t\big)}=0,
\end{equation}
with $A>3$ such that 
\begin{equation}\label{choix}
\frac{c_1}{A}<\rho-\sigma,
\end{equation} (take $\sigma=0$ for Theorem \ref{rigidite}) and
\begin{equation}\label{scalingparlim}
\forall\;i\in\{1,\dots,N\},\:\exists\;c_i^+\in\RR^*_+,\qquad \lim_{t\to+\infty}c_i(t)=c_i^+;
\end{equation}
\item control on the modulation parameters \cite[proof of Lemma 1]{tsai}: more precisely, there exists $K>0$ such that for all $t$ large enough,
\begin{equation}\label{translpar}
\forall\;i\in\{1,\ldots,N-1\},\quad
\rho_{i+1}(t)-\rho_i(t)\geq\delta t,
\end{equation}
and for all $i\in\{1,\ldots,N\}$,
\begin{equation}\label{scalingpar}
|c_i(t)-c_i|+\|\epsilon(t)\|_{H^1}\leq \frac{\delta}{K+1},
\end{equation}
\begin{equation}\label{control}
|\rho_i'(t)-c_i(t)|\leq K\left(\int_\RR\epsilon^2(t,x)e^{-\sqrt{c_1}|x-\rho_i(t)|}\;dx\right)^{\frac{1}{2}}.
\end{equation}
\end{enumerate}

\begin{Rq}
The preceding choices of $A$, $K$, and of the functions $t\mapsto c_i(t)$ as defined before are possible, provided $\alpha$ is sufficiently small.\\
\noindent Note also that estimate (\ref{scalingpar}) and assertion (\ref{scalingparlim}) guarantee that $0<c_1^+<\cdots<c_N^+$ (due to the choice of $\delta$).
\end{Rq}

\noindent\textit{Step 2: Convergence of $u(t)-\sum_{i=1}^NQ_{c_i^+}(\cdot-\rho_i(t))$ as $t\to +\infty$} \\

\begin{Lem}
We have $$\left\|u(t)-\sum_{i=1}^NQ_{c_i^+}(\cdot-\rho_i(t))\right\|_{H^1}\to 0,\qquad \text{as } t\to +\infty.$$
\end{Lem}

This lemma follows immediately from Claim \ref{rightCV} and Claim \ref{leftCV} below. We begin with this first observation.

\begin{Claim}\label{rightCV}
We have \begin{equation}
\lim_{t\to+\infty}\left\|u(t)-\sum_{i=1}^NQ_{c_i^+}(\cdot-\rho_i(t))\right\|_{H^1\big(x>\frac{c_1}{A}t\big)}=0.
\end{equation}
\end{Claim}

Let us justify this fact. \\
Using the triangular inequality and taking into account (\ref{right}), it suffices in fact to see that for all $i\in\{1,\ldots,N\}$, \begin{equation}\label{limQcj}
\lim_{t\to+\infty}\big\|Q_{c_i(t)}(\cdot-\rho_i(t))-Q_{c_i^+}(\cdot-\rho_i(t))\big\|_{H^1\big(x>\frac{c_1}{A}t\big)}=0.
\end{equation}
But the quantity $\big\|Q_{c_i(t)}(\cdot-\rho_i(t))-Q_{c_i^+}(\cdot-\rho_i(t))\big\|_{H^1\big(x>\frac{c_1}{A}t\big)} $ is bounded by $\|Q_{c_i(t)}-Q_{c_i^+}\|_{H^1}$ which tends to 0 as $t$ tends to $+\infty$. We recall indeed that the map 
\[ 
\RR^*_+ \rightarrow H^1(\RR), \quad
c \mapsto Q_c \] 
is continuous by application of Lebesgue's dominated convergence theorem. Hence (\ref{limQcj}) holds and Claim \ref{rightCV} is proved. \\

Due to the assumption of non-dispersion made in Theorem \ref{rigidite}, we claim moreover: 

\begin{Claim}\label{leftCV}
We have \begin{equation}
\lim_{t\to+\infty}\left\|u(t)-\sum_{i=1}^NQ_{c_i^+}(\cdot-\rho_i(t))\right\|_{H^1\big(x\leq\frac{c_1}{A}t\big)}=0.
\end{equation}
\end{Claim}

In what follows, we prove actually that each quantity $\|u(t)\|_{H^1\big(x\leq\frac{c_1}{A}t\big)}$ and \newline \noindent $\|Q_{c_i^+}(\cdot-\rho_i(t))\|_{H^1\big(x\leq\frac{c_1}{A}t\big)}$ for $i\in\{1,\ldots,N\}$ tends to $0$ as $t$ tends to $+\infty$. \\

\begin{enumerate}
\item Proof of $\|u(t)\|_{H^1\big(x\leq\frac{c_1}{A}t\big)}\underset{t\to+\infty}{\longrightarrow}0$. \\
\noindent Let $\varepsilon>0$. By (\ref{L2compactnessprop}) or (\ref{L2compactnessprop'}), there exists $R_\varepsilon>0$ such that for all $R\geq R_\varepsilon$,
$$\forall\;t\ge 0,\qquad \int_{x<(\rho-\sigma)t-R}u^2(t,x)\;dx\leq\frac{\varepsilon}{2}.$$
Now, by means of Proposition \ref{pointwise}, there exist $\kappa>0$ and $K_1>0$ such that for all $t\geq 0$, 
$$\forall\;x\leq (\rho-\sigma)t,\qquad |u_x(t,x)|\leq K_1e^{-\kappa|x-(\rho-\sigma)t|}.$$
Pick $R\geq R_\epsilon$ such that $\displaystyle K_1^2e^{-\kappa R}\int_\RR e^{-\kappa|x|}\;dx\leq \frac{\varepsilon}{2}$. \\
For $t$ large enough, $\frac{c_1}{A}t<(\rho-\sigma) t-R$ due to (\ref{choix}), and therefore
$$\int_{x\leq\frac{c_1}{A}t}u^2(t,x)\;dx\leq\frac{\varepsilon}{2},$$
and 
\begin{align*}
\displaystyle\int_{x\leq\frac{c_1}{A}t}u_x^2(t,x)\;dx&\leq\displaystyle\int_{x\leq (\rho-\sigma) t-R}u_x^2(t,x)\;dx \leq\displaystyle K_1\int_{x\leq (\rho-\sigma)t-R}e^{-\kappa R}|u_x(t,x)|\;dx\\
&\leq\displaystyle K_1^2e^{-\kappa R}\int_{x\leq (\rho-\sigma)t-R}e^{-\kappa|x-(\rho-\sigma)t|}\;dx \\
& \leq\displaystyle K_1^2e^{-\kappa R}\int_\RR e^{-\kappa|x-(\rho-\sigma) t|}\;dx \leq\frac{\varepsilon}{2}.
\end{align*}

As a consequence, for $t$ large enough, $\|u(t)\|_{H^1\big(x\leq\frac{c_1}{A}t\big)}^2\leq\varepsilon.$ 

\item Proof of $\|Q_{c_i^+}(\cdot-\rho_i(t))\|_{H^1\big(x\leq\frac{c_1}{A}t\big)}\underset{t\to+\infty}{\longrightarrow}0$.
\\
\noindent Notice first that, recalling (\ref{scalingpar}) and (\ref{control}), we have for $t$ large enough  
\begin{align*}
|\rho_i'(t)-c_i|&\leq\displaystyle |\rho_i'(t)-c_i(t)|+|c_i(t)-c_i| \leq K\|\epsilon(t)\|_{L^2}+|c_i(t)-c_i|\\
&\leq(K+1)\big(\|\epsilon(t)\|_{L^2}+|c_i(t)-c_i|\big) \leq\frac{c_1}{2}.
\end{align*}
\noindent In particular, for $t$ large enough, 
\begin{equation}\label{ineq_rho_i_der}
\rho_i'(t)\geq \frac{c_1}{2}.
\end{equation}
By integration of the preceding inequality, we deduce that for large values of $t$, $\rho_i(t)\geq \frac{c_1}{3}t$. Thus, for these values,
\begin{equation}\label{param}
\rho_i(t)-\frac{c_1}{A}t\geq c_1\left(\frac{1}{3}-\frac{1}{A}\right)t,\end{equation}
with $\displaystyle\frac{1}{3}-\frac{1}{A}>0$. \\
\noindent Due to the exponential decay property of the integrable functions $Q_{c_i^+}$ and $Q'_{c_i^+}$, we deduce then from (\ref{param}) that 
$\|Q_{c_i^+}(\cdot-\rho_i(t))\|_{H^1\big(x\leq\frac{c_1}{A}t\big)}\underset{t\to+\infty}{\longrightarrow}0$.
\end{enumerate}

Now, it follows from Claim \ref{rightCV} and Claim \ref{leftCV} that
\begin{equation}\label{CV_zero}
\lim_{t\to+\infty}\left\|u(t)-\sum_{i=1}^NQ_{c_i^+}(\cdot-\rho_i(t))\right\|_{H^1}=0.
\end{equation} 

\noindent\textit{Step 3: Refinement of \eqref{CV_zero}}

\begin{proposition}[Improvement of the $H^1$-convergence for asymptotic $N$-soliton like solutions]\label{improvement}
Let $u\in\mathscr{C}(\RR,H^1(\RR))$ be a solution of (gKdV) and let $0<c_1<\cdots<c_N$. Assume the existence of $T_0\ge 0$, $\delta_0>0$, and N functions $x_1,\ldots,x_N:\RR\rightarrow\RR$ of class $\mathscr{C}^1$ satisfying for all $t\ge T_0$,
\begin{equation}\label{mindelta}
\forall\;i=1,\dots,N-1,\quad 
x_{i+1}(t)-x_{i}(t)\geq \delta_0 t \qquad \text{and}\qquad \forall\;i=1,\dots,N,\quad x_i'(t)\ge\delta_0,
\end{equation}
and such that \begin{equation}\label{lim0}
\lim_{t\to+\infty}\left\|u(t)-\sum_{i=1}^NQ_{c_i}(\cdot-x_i(t))\right\|_{H^1}=0.
\end{equation}
Then there exist $C>0$ and $y_1,\ldots,y_N\in\RR$ such that 
\begin{equation}\label{limexp}
\forall\;t\geq 0,\qquad \left\|u(t)-\sum_{i=1}^NQ_{c_i}(\cdot-c_it-y_i)\right\|_{H^1}\leq Ce^{-\frac{1}{8}\delta_0^{\frac{3}{2}} t}.
\end{equation}
\end{proposition}

It was first observed by Martel \cite[Proposition 4]{martel} that multi-solitons in the sense of Definition \ref{def:multi_sol} do actually converge exponentially fast to their profile: this was a key to proving uniqueness of multi-solitons, in the $L^2$-subcritical case. In the above Proposition \ref{improvement} we further refine this observation, by noticing that the conclusion still holds even if one gives some freedom to the center of mass of the soliton $x_i(t)$ (instead of \eqref{mindelta}, the assumption in \cite[Proposition 4]{martel} was $x_i(t) = c_i t + y_i$).

The proof of Proposition \ref{improvement} follows the lines of \cite{martel} and is postponed to the Appendix; we go on assuming it holds. \\

Given \eqref{translpar}, \eqref{ineq_rho_i_der}, and \eqref{CV_zero}, we just have to apply the previous proposition (with $x_i$ replaced by $\rho_i$ and $\delta_0$ by $\delta$ defined in Step 1) to conclude that $u$ is a multi-soliton. In other words, there exist $x_1^+,\dots,x_N^+\in\RR$ such that 
\begin{equation}
\forall\;t\geq 0,\qquad \left\|u(t)-\sum_{i=1}^NQ_{c_i^+}(\cdot-c_i^+t-x_i^+)\right\|_{H^1}\leq Ce^{-\frac{1}{8}\delta^{\frac{3}{2}} t}.
\end{equation}
We recall then from Martel \cite[proof of Proposition 5]{martel} that for all $s\in\NN^*$, $u\in\mathscr{C}([0,+\infty),H^s(\RR))$ and there exists $\tilde{C}_s\geq 0$ such that
\begin{equation}\label{strong_CV}
\forall\;t\geq 0,\qquad\left\|u(t)-\sum_{i=1}^NQ_{c_i^+}(\cdot-c_i^+t-x_i^+)\right\|_{H^s}\leq \tilde{C}_se^{-\frac{1}{32}\delta^{\frac{3}{2}} t}.
\end{equation} 
This concludes the proof of Theorem \ref{rigidite}. \\

\noindent \textit{Step 4: Proof of smoothness and exponential decay of $u$} \\

\noindent Apply Proposition \ref{pointwise2} with $a(t):=(\rho-\sigma)t$ and $b(t):=(\rho+\sigma)t$ to obtain $u\in\mathscr{C}^\infty([0,+\infty)\times\RR)$ and for each $t\geq 0$,
$$
\forall\;x\leq (\rho-\sigma)t, \qquad\big|\partial^s_xu(t,x)\big|\leq K_se^{-\gamma|x-(\rho-\sigma)t|},
$$
where $K_s,\gamma>0$ are independent of $t$ and $x$. \\
Under the global non-dispersion assumption of Corollary \ref{rigidite'}, which we take as granted from now on, we have also
$$\forall\;x\geq (\rho+\sigma)t,\qquad \big|\partial^s_xu(t,x)\big|\leq K_se^{-\gamma(x-(\rho+\sigma)t)}.$$ (See Remark \ref{est_left}.)

At this stage and as we explain just below, the desired exponential decay estimate (\ref{expdecay}) follows from the strong property (\ref{CVexp}) or \eqref{strong_CV}. We distinguish three cases, depending on the position of $x$ with respect to $\pm 2(\rho+\sigma)t$; the moral being that (\ref{CVexp}) implies the expected pointwise estimate in each region $|x|\leq\zeta t$ (with an exponential decay rate depending on $\zeta$) and even if it means taking $\zeta$ large enough and reducing the decay rate $\gamma$, one can propagate the control by $e^{-\gamma|x-c_N^+t|}$ (respectively $e^{-\gamma|x-c_1^+t|}$) to the region $x>(\rho+\sigma)t$ (respectively $x<(\rho-\sigma)t$). \\

\noindent Let $t\geq 0$ and $\tilde{\theta}:=\frac{1}{32}\delta^{\frac{3}{2}}$, for each $s\in\NN$, there exists $\tilde{K}_s>0$ such that $$\left\|\partial^s_x\left(u(t)-\sum_{i=1}^NQ_{c_{i}^+}(\cdot-c_i^+t-x_i^+)\right)\right\|_{L^\infty}\leq \tilde{K}_se^{-\tilde{\theta} t}.$$

\noindent \textit{Case 1:} $|x|\leq 2(\rho+\sigma)t$. We have
\[ |x-c_N^+t|\leq (2(\rho+\sigma)+c_N^+)t \quad \text{that is,} \quad \frac{\tilde{\theta}}{2(\rho+\sigma)+c_N^+}|x-c_N^+t|\leq \tilde{\theta} t, \]
 and thus 
\begin{equation}
\left\|\partial^s_x\left(u(t)-\sum_{i=1}^NQ_{c_{i}^+}(\cdot-c_i^+t-x_i^+)\right)\right\|_{L^\infty}\leq \tilde{K_s}e^{-\frac{\tilde{\theta}}{2(\rho+\sigma)+c_N^+}|x-c_N^+t|}.
\end{equation}
Consequently for $t\geq 0$ and $|x|\leq 2(\rho+\sigma)t$, using the triangular inequality and the exponential decay of $\partial_x^sQ_{c_i^+}$, we obtain
\begin{align}
\displaystyle|\partial_x^su(t,x)|&\leq\displaystyle\sum_{i=1}^N\left|\partial_x^sQ_{c_i^+}(x-c_i^+t-x_i^+)\right|+\displaystyle\left\|\partial^s_x\left(u(t)-\sum_{i=1}^NQ_{c_{i}^+}(\cdot-c_i^+t-x_i^+)\right)\right\|_{L^\infty}\notag\\
&\leq\displaystyle \tilde{\tilde{K_s}}\sum_{i=1}^Ne^{-\tilde{\gamma}|x-c_i^+t|},
\end{align}

where $\tilde{\gamma}:=\min\Big\{\sqrt{c_1^+},\frac{\tilde{\theta}}{2(\rho+\sigma)+c_N^+}\Big\}$. \\

\noindent \textit{Case 2:} $x\geq 2(\rho+\sigma) t$. Let us rewrite this as $x-(\rho+\sigma)t\geq \frac{1}{2}x.$ In particular 
\[ x-(\rho+\sigma)t\geq \frac{1}{2}\left(x-c_N^+t\right) \]
so that for $x\geq 2(\rho+\sigma)t$,
\begin{align}
\big|\partial^s_xu(t,x)\big|&\leq\displaystyle K_se^{-\gamma(x-(\rho+\sigma)t)} \leq K_se^{-\frac{\gamma}{2}(x-c_N^+t)}.
\end{align}

\noindent \textit{Case 3:} $x\leq -2(\rho+\sigma)t$. Arguing similarly as before, we have then
\begin{align}
\big|\partial^s_xu(t,x)\big|&\leq\displaystyle K_se^{-\gamma((\rho-\sigma)t-x)} \leq\displaystyle K_se^{-\frac{\gamma}{L}(c_1^+t-x)},
\end{align}
for $L>2$ chosen such that $\frac{L(\rho-\sigma)-c_1^+}{L-1}>-2(\rho+\sigma)$. 

\noindent Set finally $\theta:=\min\left\{\frac{\gamma}{L},\tilde{\gamma}\right\}$ to obtain (\ref{expdecay}) in Corollary \ref{rigidite'}.

\section{The integrable cases: proofs of Theorems {\normalfont\textbf{\ref{kdv}}} and {\normalfont\textbf{\ref{mkdv}}}}

\subsection{Non dispersive solutions of the Korteweg-de Vries equation}\label{proof_th_kdv}

The strategy to prove Theorem \ref{kdv} takes inspiration in \cite[Proof of Theorem 2]{laurent}. We use the following result of Eckhaus and Schuur \cite[Section 5]{eckhaus}, which is also a consequence of a generalized version by Schuur \cite[Chapter 2, Theorem 7.1 and (7.23)]{schuur}.

\begin{theorem}[Eckhaus and Schuur \cite{eckhaus}; Schuur \cite{schuur}]\label{dec_schuur}
Let $p=2$ and $u_0\in\mathscr{C}^4(\RR)$ be such that for some $C_0>0$, for all $k=0,\dots,4$, and for all $x\in\RR$,
\begin{equation}\label{dec_exp}
\left|\frac{\partial^ku_0(x)}{\partial x^k}\right|\leq C_0|x|^{-11}.
\end{equation} Let $u$ be the corresponding global solution of (KdV).

 Then there exists a solution $u_d$ which is a multi-soliton or zero such that for all $\beta>0$, there exist $\gamma=\gamma(\beta)>0$ and $K\geq 0$ such that for all $t>0$
\begin{equation}\label{prop_asympt}
\|u(t)-u_d(t)\|_{L^\infty(x>\beta t)}+\|u(t)-u_d(t)\|_{L^2(x>\beta t)}\leq Kt^{-\frac{1}{3}}.
\end{equation}
\end{theorem}

\begin{proof}[Proof of Theorem \ref{kdv}]
Due to the assumption on $u_0$ in Theorem \ref{kdv}, we can apply Theorem \ref{dec_schuur} and we obtain a solution $u_d$ of (KdV) as above which fulfills \eqref{prop_asympt}. We claim first that $u_d$ is not the trivial solution. Otherwise, with $\beta:=\frac{\rho}{2}>0$, we would have $\|u(t)\|_{L^2(x>\beta t)}=\mathrm{O}(t^{-\frac{1}{3}})$ as $t$ tends to $+\infty$. On the other hand, by the non dispersion assumption and namely by Proposition \ref{pointwise}, \begin{equation}\label{est_left_L^infty}
\forall\;x\leq\beta t,\qquad\left|u(t,x)\right|\leq Ce^{-\gamma|x-\rho t|},\end{equation}
so that 
\begin{equation}\label{est_left_L^2}
\|u(t)\|_{L^2(x\leq\beta t)}\leq Ce^{-\gamma\beta t}.
\end{equation}
Then, we would obtain that $$\|u(t)\|_{L^2}=\|u(t)\|_{L^2(x\leq\beta t)}+\|u(t)\|_{L^2(x>\beta t)}\to 0\qquad \text{as } t\to +\infty,$$ hence conclude that $\|u_0\|_{L^2}=0$ by the mass conservation law. This contradicts our assumption in Theorem \ref{kdv}. \\
Thus there exist $N\geq 1$, $0<c_1<\dots<c_N$, $x_1^+,\dots,x_N^+\in\RR$, and a possibly smaller $\gamma>0$ such that \begin{equation}\label{CV_H1}
\left\|u_d(t)-\sum_{i=1}^NR_{c_i,x_i^+}(t)\right\|_{H^1}=\mathrm{O}\left(e^{-\gamma t}\right),\quad \text{as } t\to +\infty.
\end{equation}
\begin{Claim}\label{claim_partiel}
We have
$$ \left\|u(t)-\sum_{i=1}^NR_{c_i,x_i^+}(t)\right\|_{L^2}=\mathrm{O}\left(t^{-\frac{1}{3}}\right), \qquad\text{as } t\to +\infty.$$
\end{Claim}

\begin{proof}[Proof of Claim \ref{claim_partiel}]
Consider $\beta\in\left(0,\min\{c_1,\rho\}\right)$ so that \eqref{prop_asympt} is guaranteed for some $\gamma>0$ and so that (by the non dispersion assumption and the sech-shaped profiles of the solitons $R_{c_i,x_i^+}$)
$$\|u(t)\|_{L^2(x\leq \beta t)}+\sum_{i=1}^N\left\|R_{c_i,x_i^+}(t)\right\|_{L^2(x\leq\beta t)}=\mathrm{O}\left(e^{-\gamma t}\right)$$ even if it means reducing $\gamma>0$. We perform then
\begin{align*}
\|u(t)-u_d(t)\|_{L^2}&= \|u(t)-u_d(t)\|_{L^2(x\leq\beta t)}+\|u(t)-u_d(t)\|_{L^2(x>\beta t)}\\
&\leq\|u(t)\|_{L^2(x\leq\beta t)}+\|u_d(t)\|_{L^2(x\leq\beta t)}+\mathrm{O}\left(t^{-\frac{1}{3}}\right)\\
&\leq \|u_d(t)\|_{L^2(x\leq\beta t)}+\mathrm{O}\left(t^{-\frac{1}{3}}+e^{-\gamma t}\right)\\
&\leq \left\|u_d(t)-\sum_{i=1}^NR_{c_i,x_i^+}(t)\right\|_{H^1}+\sum_{i=1}^N\left\|R_{c_i,x_i^+}(t)\right\|_{L^2(x\leq\beta t)}+\mathrm{O}\left(t^{-\frac{1}{3}}\right)\\
&= \mathrm{O}\left(t^{-\frac{1}{3}}\right),
\end{align*} by the embeddings $H^1(\RR)\hookrightarrow H^1(x\leq\beta t)\hookrightarrow L^2(x\leq\beta t) $ and by \eqref{CV_H1}. By means of the triangular inequality and once again \eqref{CV_H1}, we deduce the expected estimate in Claim \ref{claim_partiel}.
\end{proof}

We are now able to finish the proof of Theorem \ref{kdv}. Indeed, let us make the following key observation.
 
\begin{Claim}\label{H2_bounded}
The solution $u$ belongs to $L^\infty([0,+\infty),H^2(\RR))$.
\end{Claim}

\begin{proof}[Proof of Claim \ref{H2_bounded}]
This is an immediate consequence of the following conservation law for the KdV equation 
\begin{equation}\label{cons_law}
\frac{d}{dt}\int_\RR\left\{\left(\partial_x^2u\right)^2-\frac{10}{3}\left(\partial_xu\right)^2u+\frac{5}{9}u^4\right\}(t,x)\;dx=0,
\end{equation}
of the Sobolev embedding $H^1(\RR)\hookrightarrow L^\infty(\RR)$, and from the fact that $u$ belongs to $L^\infty(\RR,H^1(\RR))$.
\end{proof}

As a consequence of Claim \ref{H2_bounded}, $v$ belongs also to $L^\infty([T_1,+\infty),H^2(\RR))$. Then, integrating by parts and using the Cauchy-Schwarz inequality and Claim \ref{claim_partiel}, we infer that
\begin{align*}
\int_\RR\left(\partial_xv\right)^2(t)\;dx&=-\int_\RR v(t)\partial_x^2v(t)\;dx \leq \|v(t)\|_{L^2}\|v(t)\|_{H^2} \leq Ct^{-\frac{1}{3}},
\end{align*}
from which it results that 
\begin{equation}
\left\|u(t)-\sum_{i=1}^NR_{c_i,x_i^+}(t)\right\|_{H^1}\rightarrow 0,\qquad \text{as }t\rightarrow +\infty.
\end{equation}

Hence $u$ is a multi-soliton in $+\infty$. By means of the well-known theory concerning multi-solitons of the KdV equation (see for instance Miura \cite{miura}), we deduce that $u$ is also a multi-soliton in $-\infty$. This ends the proof of Theorem \ref{kdv}. 
\end{proof}

\subsection{Non dispersive solutions of the modified Korteweg-de Vries equation}

Theorem \ref{mkdv} is obtained by using the same strategy as that developed in the previous subsection. Thus we will only sketch its proof. 

As for the KdV case, we apply first the following decomposition result, obtained from \cite[Chapter 5, Theorem 5.1]{schuur} and from \cite[Theorem 1.10]{chen} where a more precise version can be found.

\begin{theorem}[Schuur \cite{schuur}, Chen and Liu \cite{chen}]\label{dec_schuur_mkdv}
Let $p=3$ and $u_0\in\mathscr{C}^4(\RR)$ be such that for some $C_0>0$, for all $k=0,\dots,4$, and for all $x\in\RR$,
\begin{equation}
\left|\frac{\partial^ku_0(x)}{\partial x^k}\right|\leq C_0|x|^{-11},
\end{equation} and be generic as in Theorem \ref{mkdv}, with scattering data \eqref{scattering_set}. 
Let $u$ be the corresponding global solution of (mKdV). 

Then there exist signs $\epsilon_i=\pm 1$, $i=1,\dots, N_1$, and parameters $x_{0,i}$, $i=1,\dots,N_1$, and $x_{1,j}$, $x_{2,j}$, $j=1,\dots,N_2$, such that for all $ v_+>0$ and $v_-<0$, there exists $K\geq 0$ such that for all $t>0$, denoting 
$$P(t):=\sum_{i=1}^{N_1}\epsilon_i R_{2c_i,x_{0,i}}(t)+\sum_{j=1}^{N_2}B_{\sqrt{2}\alpha_j,\sqrt{2}\beta_j,x_{1,j},x_{2,j}}(t),$$ we have
\begin{equation}
\left\|u(t)-P(t)\right\|_{L^\infty(x>v_+ t)}+ \|u(t)-P(t)\|_{L^2(x>v_+t)}\leq Kt^{-\frac{1}{3}},
\end{equation}
and
\begin{equation}\label{estimate_left}
\left\|u(t)-P(t)\right\|_{L^\infty(x<v_- t)}\leq Kt^{-\frac{1}{2}},
\end{equation}
\end{theorem}

Then, the non dispersion assumption \eqref{L2compactnessprop} in Theorem \ref{mkdv} shows that $N_1+N_2\ge 1$. Since the profiles of the breathers are sech-shaped, due to \eqref{L2compactnessprop} and \eqref{estimate_left}, we deduce that the breathers have positive (envelope) velocities. 

Now, proceeding as in the proof of Claim \ref{claim_partiel}, we obtain in fact that  
$$\|u(t)-P(t)\|_{L^2}=\mathrm{O}\left(t^{-\frac{1}{3}}\right).$$

Moreover, (mKdV) admits conservation laws of orders 2, 3, and 4 in the spirit of \eqref{cons_law}, which shows that $u$ belongs to $L^\infty([0,+\infty),H^{4}(\RR))$. Proceeding similarly to subsection \ref{proof_th_kdv}, we obtain that
$$\left\|u(t)-P(t)\right\|_{H^2}\to 0,\qquad \text{as } t\to +\infty.$$

Finally, by the uniqueness and smoothness results and the estimates in higher Sobolev spaces proven by Semenov \cite{semenov} as far as multi-breathers are concerned, we deduce that $u$ belongs to $\mathscr{C}([0,+\infty),H^s(\RR))$ and that there exist $\gamma>0$ and positive constants $C_s$ such that for all $s\in\NN$, 
$$\left\|u(t)-P(t)\right\|_{H^s}\le C_se^{-\gamma t},\qquad \text{as } t\to +\infty.$$
This finishes proving Theorem \ref{mkdv}.

\section{Appendix: Proof of Proposition \normalfont\textbf{\ref{improvement}}}

The proof follows the same lines as that of Proposition 3 and paragraph 3.2 in Martel \cite{martel} for the $L^2$-subcritical and critical cases, and that of Lemma 4.1 in Combet \cite{combetgkdv} for the supercritical case. For the sake of simplicity and for the reader's convenience, we present here the essential ideas and also the changes in the $L^2$-subcritical case only. 

\begin{Rq}
We mention that in the $L^2$-critical and supercritical cases, the proof is basically changed in terms of the coercivity property we use to control the modulation function $\epsilon$ defined below in Step 1. The monotonicity properties of local mass and energy obtained in Step 2 are still valid in these cases. \\
\noindent Concerning the critical case, the idea is to modulate the scaling parameter in addition to the translation parameter so as to ensure a second orthogonality condition satisfied by $\epsilon$, namely $\int_\RR \epsilon(t)\tilde{R}_j(t)^3\;dx=0$, and then to apply a localized version of the coercivity property available in this case, which leads to:
\begin{equation}
\exists\;\lambda_0>0,\:\forall\;t,\qquad\|\epsilon(t)\|_{H^1}^2\leq \lambda_0\mathcal{H}(t),
\end{equation} with $\mathcal{H}$ defined in Step 3. \\
\noindent In the supercritical case, it is known from Pego and Weinstein \cite{pego} that, considering the standard linearized operator $L$ on $H^1(\RR)$ defined by $Lv:=-\partial_x^2v+v-pQ^{p-1}v$, the composed operator $L\partial_x$ has two eigenfunctions $Z^+$ and $Z^-$ related by $Z^-(x)=Z^+(-x)$, which decay exponentially, and such that $L\partial_xZ^{\pm}=\pm e_0Z^{\pm}$ for some $e_0>0$. In this case, we only have to make modifications in Step 3 by using this time
\begin{equation}
\exists\;\lambda_0>0,\:\forall\;t,\qquad\|\epsilon(t)\|_{H^1}^2\leq \lambda_0\mathcal{H}(t)+\frac{1}{\lambda_0}\sum_{i,\pm}\left(\int_\RR\epsilon(t)\tilde{Z}_i^{\pm}(t)\;dx\right)^2,
\end{equation} where $\displaystyle \tilde{Z}_i^{\pm}(t):=Z_i^{\pm}(\cdot-x_i(t)-y_i(t))$ and $\displaystyle Z_i^{\pm}(x):=c_i^{-\frac{1}{2}}Z^{\pm}\left(c_i^{\frac{1}{2}}x\right)$. \\
(The functions $y_i$ are defined in Lemma \ref{modul} below.)
\noindent The control of $\int_\RR\epsilon(t)\tilde{Z}_i^{\pm}(t)\;dx$ by a function of $t$ which decreases with exponential speed follows the strategy of Combet (for full details, see \cite[paragraph 4.1 Step 4]{combetgkdv}).
\end{Rq}

\noindent \textit{Step 1: Set up of a modulation argument}\\

Set $\nu:=\min\{c_1,\delta_0\}$. We claim the following 

\begin{Lem}\label{modul}
There exist $T\geq 0$ and $\alpha_1\in (0,1]$ such that for all $\tilde{\alpha}\leq\alpha_1$, the following holds. There exist unique $\mathscr{C}^1$ functions $y_i:[T,+\infty)\rightarrow\RR$ such that defining
\begin{equation}\label{def:epsilon}
\epsilon:=u-\sum_{i=1}^N\tilde{R_i},
\end{equation}
where $\tilde{R}_i(t,x):=Q_{c_i}\big(x-x_i(t)-y_i(t)\big)$, we have for all $t\ge T$,
\begin{equation}\label{orthogonality_prop}
\forall\;i\in\{1,\ldots,N\},\quad\int_{\RR}\epsilon\big(\tilde{R_i}\big)_x(t)\;dx=0.
\end{equation}
In addition, there exists $K>0$ such that for all $t\geq T$, for all $i\in\{1,\ldots,N\}$, 
\begin{equation}\label{est1}
\|\epsilon(t)\|_{H^1}+\sum_{i=1}^N|y_i(t)|\leq K\tilde{\alpha},
\end{equation}
\begin{equation}\label{est2}
|x_i'(t)+y_i'(t)-c_i|\leq K\left(\int_\RR\epsilon^2(t)e^{-\sqrt{\nu}|x-x_i(t)|}\;dx\right)^{\frac{1}{2}}+Ke^{-\frac{1}{4}\nu^{\frac{3}{2}}t}.
\end{equation}
\end{Lem}

\begin{proof}
Recall that the proof of existence and uniqueness of the functions $y_i(t)$ is based on the implicit function theorem. We refer to \cite[proof of Lemma 8]{tsai} and also to \cite[paragraph 2.3]{mm_instabilite} for a complete proof in the case of one soliton. Moreover, estimate (\ref{est2}) which involves $\nu\le c_1$ is obtained formally by writing the equation of $\epsilon$, that is
\begin{equation}
\epsilon_t+\partial^3_x\epsilon=\sum_{i=1}^N(x_i'+y_i'-c_i)(\tilde{R}_i)_x-\left(\left(\epsilon+\sum_{i=1}^N\tilde{R}_i\right)^p-\sum_{i=1}^N\tilde{R}_i^p\right)_x,
\end{equation} by multiplying it by $(\tilde{R}_i)_x$, and by using the following properties:
\begin{equation}
0=\frac{d}{dt}\int_{\RR}\epsilon(\tilde{R}_i)_x\;dx=\int_\RR\partial_t\epsilon\partial_x\tilde{R}_i\;dx-(x_i'+y_i')\int_\RR\epsilon\partial^2_x\tilde{R}_i\;dx;
\end{equation}
$\forall\;i\neq j,\:\forall\;t\geq T_2$,
\begin{equation}\label{ineq_Rj}
|\tilde{R}_i(t,x)|+|\partial_x\tilde{R}_i(t,x)|\leq Ce^{-\sqrt{\nu}|x-x_i(t)|};
\end{equation}
\begin{equation}\label{ineg_interaction}
\int_\RR\left\{\tilde{R}_i(t,x)\tilde{R}_j(t,x)+|\partial_x\tilde{R}_i(t,x)\partial_x\tilde{R}_j(t,x)|\right\}\;dx\leq Ce^{-\frac{\nu^{\frac{3}{2}}}{2}t}.
\end{equation}
Note that \eqref{ineg_interaction} is a consequence of the decoupling assumption \eqref{mindelta}. We refer to \cite{martel} and the references therein for more details. 
\end{proof}

\noindent \textit{Step 2: Monotonicity properties for localized mass and some modified energy of $u$} \\

Let $\psi:x\mapsto\frac{2}{\pi}\mathrm{Arctan}\left(e^{-\frac{\sqrt{\nu}}{2}x}\right)$  be defined on $\RR$ so that for all $x\in\RR$,
\[ \psi '(x)\leq 0,\quad \quad |\psi '(x)|\leq\frac{\sqrt{\nu}}{\pi}e^{-\frac{\sqrt{\nu}}{2}|x|},\quad\quad |\psi^{(3)}(x)|\leq\frac{\nu}{4}|\psi '(x)|,\] (we recall $\nu=\min\{c_1,\delta_0\}$).
Then define on $\RR^+\times\RR$:
\begin{equation}
\forall\;i\in\{1,\ldots,N-1\},\quad \psi_i:(t,x)\mapsto \psi\left(x-\frac{x_i(t)+x_{i+1}(t)}{2}\right), \quad \psi_N:(t,x)\mapsto 1
\end{equation}
and also 
\begin{equation}
\phi_1:=\psi_1, \quad
\forall\;i\in\{2,\ldots,N-1\},\quad \phi_i:=\psi_i-\psi_{i-1}, \quad
\phi_N:=1-\psi_{N-1}.
\end{equation}

\begin{Rq}
Note that by definition and by (\ref{mindelta}), for $t>0$, $\phi_i(t)$ takes values close to 1 in a neighborhood of $x_i(t)$ and takes values close to 0 around $x_j(t)$ for $j\neq i$.
\end{Rq}

Take $\kappa\in\big(0,\frac{c_1}{4}\big)$ and consider now for all $i\in\{1,\ldots,N-1\}$ the following quantities: 
\begin{itemize}
\item (localized mass of $u$ at the left of $\frac{x_i(t)+x_{i+1}(t)}{2}$)
\begin{equation}
\mathcal{M}_i(t):=\int_\RR u^2(t,x)\psi_i(t,x)\;dx
\end{equation}
\item (modified localized energy of $u$ at the left of $\frac{x_i(t)+x_{i+1}(t)}{2}$) 
\begin{equation}
\tilde{\mathcal{E}}_i(t):=\int_\RR \left(\frac{1}{2}u_x^2-\frac{1}{p+1}u^{p+1}+\kappa u^2\right)\psi_i(t,x)\;dx.
\end{equation}
\end{itemize}
Let also \begin{equation}
\mathcal{M}_N(t):=\int_\RR u^2(t,x)\psi_N(t,x)\;dx
\end{equation} (which is nothing but the mass of $u$)
and \begin{equation}
\tilde{\mathcal{E}}_N(t):=\int_\RR \left(\frac{1}{2}u_x^2-\frac{1}{p+1}u^{p+1}+\kappa u^2\right)\psi_N(t,x)\;dx
\end{equation} (which is a global quantity linked to the energy of $u$).

\begin{Rq}
The reason why we have to choose $\kappa$ small enough appears clearly in Step 3 (see Remark \ref{rk:kappa}).
\end{Rq}

\noindent We claim now a monotonicity result on the preceding quantities.

\begin{Lem}\label{lem:mono}
There exist $T_1\geq T_0$ and $K_1\geq 0$ such that for all $t\geq T_1$ and for all $i\in\{1,\ldots,N\}$,
\begin{equation}\label{prop_monotonicity}
\frac{d\mathcal{M}_i}{dt}(t)\geq -K_1e^{-\frac{\nu^{\frac{3}{2}}}{4}t}\qquad\text{and}\qquad
\frac{d\tilde{\mathcal{E}}_i}{dt}(t)\geq -K_1e^{-\frac{\nu^{\frac{3}{2}}}{4}t}.
\end{equation}
\end{Lem}

\begin{proof}
First we observe that
$$\frac{d\mathcal{M}_i}{dt}=-\int_\RR \left(3u_x^2+\frac{x_i'+x_{i+1}'}{2}u^2-\frac{2p}{p+1}u^{p+1}\right)\psi_i'\;dx+\int_\RR u^2\psi_i^{(3)}\;dx.$$
For all $\eta_0>0$, there exists $T_{\eta_0}\geq 0$ such that for all $t\geq T_{\eta_0}$,
$$\left\|u(t)-\sum_{i=1}^NQ_{c_i}(\cdot-x_i(t))\right\|_{H^1}\leq\eta_0$$ and for all $R_0>0$, for each $(t,x)\in\RR^+\times\RR$ such that $x_i(t)+R_0\leq x\leq x_{i+1}(t)-R_0$, we have 
\begin{align*}
|u(t,x)|&\leq \displaystyle\sum_{i=1}^N Q_{c_i}(x-x_i(t))+C\left\|u(t)-\sum_{i=1}^NQ_{c_i}(\cdot-x_i(t))\right\|_{H^1} \leq\displaystyle C\sum_{i=1}^Ne^{-\sqrt{c_i}R_0}+\eta_0.
\end{align*}
Thus, for $R_0$ sufficiently large and for $\eta_0>0$ small enough being fixed, we have for some $T_1>T_0$: for all $t\geq T_1$, for all $x\in[x_i(t)+R_0,x_{i+1}(t)-R_0]$, 
$$\frac{2p}{p+1}|u(t,x)|^{p-1}\leq\frac{\nu}{4}.$$
If $x>x_{i+1}(t)-R_0$ or $x<x_i(t)+R_0$, then $$\left|x-\frac{x_i(t)+x_{i+1}(t)}{2}\right|>\frac{x_{i+1}(t)-x_i(t)}{2}-R_0>\frac{\nu t}{2}-R_0.$$
Consequently, for $t\geq T_1$ and $x\notin[x_i(t)+R_0,x_{i+1}(t)-R_0]$, we obtain
\begin{align*}
\left|\psi_i'(t,x)\right|&\leq \frac{\sqrt{\nu}}{\pi}e^{-\frac{\sqrt{\nu}}{2}\big(\frac{x_{i+1}(t)-x_i(t)}{2}-R_0\big)} \leq Ce^{-\frac{\nu^\frac{3}{2}}{4}t}.
\end{align*}
We deduce that 
\begin{equation}
\begin{aligned}
\displaystyle\frac{d\mathcal{M}_i}{dt}(t)&\ge -\int_\RR\left(3u_x^2+\left(\delta_0-\frac{\nu}{4}\right)u^2\right)\psi_i'-\int_\RR\frac{\nu}{4}u^2|\psi_i'|-Ce^{-\frac{\nu^{\frac{3}{2}}}{4}t}\\
&\geq -\int_\RR\left(3u_x^2+\frac{\nu}{2}u^2\right)\psi_i'-Ce^{-\frac{\nu^{\frac{3}{2}}}{4}t} \geq\displaystyle-Ce^{-\frac{\nu^{\frac{3}{2}}}{4}t}.
\end{aligned}
\end{equation}

Similarly, we compute
\begin{equation}
\begin{aligned}
\displaystyle\frac{d\tilde{\mathcal{E}}_i}{dt}=&\:-\displaystyle\int_\RR \left[\left(u_{xx}^2+u^p\right)^2+2u_{xx}^2+\frac{x_i'+x_{i+1}'}{2}u_x^2-\frac{x_i'+x_{i+1}'}{p+1}u^{p+1}\right](\psi_i)_x\;dx\\
&\displaystyle+\int_\RR u_x^2\psi_i^{(3)}\;dx+2p\int_\RR u_x^2u^{p-1}(\psi_i)_x\;dx+\kappa\frac{d\mathcal{M}_i}{dt}\\
\geq&\displaystyle-\nu\int_\RR u_x^2(\psi_i)_x+\frac{\nu}{4}\int_\RR u_x^2(\psi_i)_x+2p\int_\RR u_x^2u^{p-1}(\psi_i)_x\\
&\displaystyle-\kappa\int_\RR 3u_x^2\psi_i'(x)-Ce^{-\frac{\nu^{\frac{3}{2}}}{4}t}+\frac{x_i'+x_{i+1}'}{2}\int_\RR \left(\frac{2}{p+1}u^{p+1}-\kappa u^2\right)(\psi_i)_x.
\end{aligned}
\end{equation}
As before, we can increase $T_1$ and reduce $\eta_0$ to have $$2p\left|\int_\RR u_x^2u^{p-1}(\psi_i)_x\right|\leq \frac{\nu}{4}\int_\RR u_x^2\left|(\psi_i)_x\right|+Ce^{-\frac{\nu^{\frac{3}{2}}}{4}t}$$ and 
$$\left|\frac{2}{p+1}\int_\RR u^{p+1}(\psi_i)_x\right|\leq \frac{\kappa}{2}\int_\RR u^2\left|(\psi_i)_x\right|.$$
Eventually, this leads to 
\begin{align*}
\displaystyle\frac{d\tilde{\mathcal{E}}_i}{dt}&\geq\displaystyle -\frac{3}{4}\nu\int_\RR u_x^2(\psi_i)_x-Ce^{-\frac{\nu^{\frac{3}{2}}}{4}t}+\frac{\kappa(x_i'+x_{i+1}')}{4}\int_\RR u^2\left|(\psi_i)_x\right|\\
&\geq\displaystyle \frac{3}{4}\nu\int_\RR u_x^2\left|(\psi_i)_x\right|+\frac{\kappa\delta_0}{2}\int_\RR u^2\left|(\psi_i)_x\right|-Ce^{-\frac{\sigma^{\frac{3}{2}}}{4}t} \ge -Ce^{-\frac{\nu^{\frac{3}{2}}}{4}t}. \qedhere
\end{align*}
\end{proof}

\noindent \textit{Step 3: A Weinstein type functional}\\

Let the functional $\mathcal{H}$ be given by \begin{equation}
\mathcal{H}:=\displaystyle\sum_{i=1}^N\frac{1}{c_i^2}\int_\RR\left\{\partial_x\epsilon^2+c_i\epsilon^2-p\tilde{R}_i^{p-1}\epsilon^2\right\}\phi_i,
\end{equation}
define \begin{equation}
\mathcal{F}:=\sum_{i=1}^N\frac{1}{c_i^2}\left\{\int_\RR\left(\frac{1}{2}u_x^2-\frac{1}{p+1}u^{p+1}\right)\phi_i+\frac{c_i}{2}\int_\RR u^2\phi_i\right\},
\end{equation}
and set $w(t):=u(t)-\sum_{i=1}^NR_i(t)$, where for all $i=1,\dots,N$,
$$R_i(t):=Q_{c_i}(\cdot-x_i(t)).$$
\noindent We gather next some properties satisfied by $\mathcal{H}$ and $\mathcal{F}$ which are essential to obtain $\|\epsilon(t)\|_{H^1}=\mathrm{O}\left(e^{-\gamma t}\right)$ as $t\to +\infty$, for some $\gamma>0$.

\begin{Lem}\label{lem:prop_step3}
We have 
\begin{enumerate}
\item (coercivity property satisfied by $\mathcal{H}$) \begin{equation}\label{eg3}
\exists\;\lambda_0>0,\:\forall\;t\geq T,\quad \|\epsilon(t)\|_{H^1}^2\leq \lambda_0\mathcal{H}(t)+\frac{1}{\lambda_0}\sum_{i=1}^N\left(\int_\RR\epsilon(t)\tilde{R}_i(t)\right)^2;
\end{equation}
\item (expansion of $\mathcal{H}$)
\begin{equation}\label{eg1}
\displaystyle\mathcal{H}=2\left(\mathcal{F}-\sum_{i=1}^N\frac{1}{c_i^2}\left\{\int_\RR\left(\frac{1}{2}(\partial_xQ_{c_i})^2-\frac{1}{p+1}Q_{c_i}^{p+1}\right)+\frac{c_i}{2}\int_\RR Q_{c_i}^2\phi_i\right\}\right)+g,
\end{equation}
\noindent where $|g(t)|\leq Ce^{-\frac{1}{4}\nu^{\frac{3}{2}}t}+C\tilde{\alpha}\|\epsilon(t)\|^2_{L^2}$;
\item (second expression for $\mathcal{F}$) 
\begin{equation}\label{eg2}
\begin{aligned}
\mathcal{F}=&\displaystyle\sum_{i=1}^{N-1}\left\{\left(\frac{1}{c_i^2}-\frac{1}{c_{i+1}^2}\right)\tilde{\mathcal{E}}_i+\left(\frac{1}{c_i}-\frac{1}{c_{i+1}}\right)\left(\frac{1}{2}-\kappa\left(\frac{1}{c_i}+\frac{1}{c_{i+1}}\right)\right)\mathcal{M}_i\right\}\\
&+\frac{1}{c_N^2}\tilde{\mathcal{E}}_N+\frac{1}{c_N}\left(\frac{1}{2}-\frac{\kappa}{c_N}\right)\mathcal{M}_N;
\end{aligned}
\end{equation}
\item (consequence of the monotonicity properties)
For all $t'\geq t\geq T_1,$ for all $i=1,\dots,N$,
\begin{equation}\label{ineqM}
\mathcal{M}_i(t)-\int_\RR Q_{c_i}^2\leq 2\int_\RR wR_i(t')+\int_\RR w^2\phi_i(t')+Ce^{-\frac{\nu^{\frac{3}{2}}}{4}t};
\end{equation}
\begin{equation}\label{ineqE}
\begin{aligned}
\tilde{\mathcal{E}}_i(t)-\int_\RR\left\{\frac{1}{2}\left(\partial_{x}Q_{c_i}\right)^2-\frac{Q_{c_i}^{p+1}}{p+1}\right\}\leq&\: C\|w(t)\|_{L^\infty}\int_\RR w^2\phi_i(t')-c_i\int_\RR wR_i(t')\\
&\displaystyle+\frac{1}{2}\int_\RR (w_x^2-pR_i^{p-1}w^2)\phi_i(t')+Ce^{-\frac{\nu^{\frac{3}{2}}}{4}t}.
\end{aligned}
\end{equation}
\end{enumerate}
\end{Lem}

\begin{proof}[Proof of Lemma \ref{lem:prop_step3}]
We only give some indications, and particularly the key ingredients. Property \eqref{eg2} is obtained by Abel transformation. Now, we focus on the other lines. \\
Estimate \eqref{eg3} is a consequence of a localized version around $\phi_i$ of the coercivity property of the linearized operator around $Q_{c_i}$ (for all $i=1,\dots,N$), which holds under the orthogonality conditions \eqref{orthogonality_prop} satisfied by $\epsilon$; we refer to \cite[proof of Lemma 4]{tsai}. \\
To prove \eqref{eg1}, one has obviously to replace $\epsilon$ by its definition \eqref{def:epsilon}. \\
To finish with, integrate the almost monotonicity properties as expressed in Lemma \ref{lem:mono} between $t$ and $t'$ and use the expression of $u(t')$ in terms of $w(t')$ in order to obtain \eqref{ineqM} and \eqref{ineqE}. \\
Let us mention furthermore that \eqref{eg3}, \eqref{eg1}, \eqref{ineqM}, and \eqref{ineqE} rely all on classical inequalities used in studying quantities which are localized near the solitons, and which write in the present context as follows:
\begin{align}
\forall\;i\neq j,&\qquad\left(R_i(t,x)+|\partial_xR_i(t,x)|\right)\phi_j(t,x)\le Ce^{-\frac{\nu^{\frac{3}{2}}}{4}t}e^{-\frac{\sqrt{\nu}}{4}|x-x_i(t)|}\\
\forall\;i,j,&\qquad\left(R_i(t,x)+|\partial_xR_i(t,x)|\right)|\partial_x\phi_j(t,x)|\le Ce^{-\frac{\nu^{\frac{3}{2}}}{4}t}e^{-\frac{\sqrt{\nu}}{4}|x-x_i(t)|} \\
\forall\;i,&\qquad R_i(t,x)\left(1-\phi_i(t,x)\right)\le Ce^{-\frac{\nu^{\frac{3}{2}}}{4}t}e^{-\frac{\sqrt{\nu}}{4}|x-x_i(t)|}. \qedhere
\end{align}

\end{proof}

Let us now explain how to conclude the proof. Note that (\ref{eg1}), (\ref{eg2}), (\ref{ineqM}), and (\ref{ineqE}) lead to: $\forall\;t'\geq t\geq T,$
\begin{equation}
\begin{aligned}
\mathcal{H}(t)\leq&\displaystyle\: Ce^{-\frac{\nu^{\frac{3}{2}}}{4}t}+|g(t)|\\
&\displaystyle+2\sum_{i=1}^{N-1}\left(\frac{1}{c_i^2}-\frac{1}{c_{i+1}^2}\right)\bigg(C\|w(t')\|_{L^\infty}\int_\RR w^2\phi_i(t')-c_i\int_\RR w(t')R_i(t')\phi_i(t')\\
&\qquad\qquad\qquad\qquad\qquad\displaystyle+\frac{1}{2}\int_\RR\left\{(\partial_xw)^2(t')-pR_i^{p-1}w^2(t')\right\}\phi_i(t')\bigg)\\
&\displaystyle+2\sum_{i=1}^{N-1}\left(\frac{1}{c_i}-\frac{1}{c_{i+1}}\right)\left(\frac{1}{2}-\kappa\left(\frac{1}{c_i}+\frac{1}{c_{i+1}}\right)\right)\left(2\int_\RR wR_i(t')+\int_\RR w^2\phi_i(t')\right)\\
&\displaystyle+\frac{2}{c_N^2}\left(C\|w(t')\|_{L^\infty}\int_\RR w^2\phi_N(t')+\frac{1}{2}\int_\RR\left\{(\partial_xw)^2-pR_N^{p-1}w^2(t')\right\}\phi_N(t')\right)\\
&\displaystyle+\frac{2}{c_N}\left(\frac{1}{2}-\frac{\kappa}{c_N}\right)\left(2\int_\RR wR_N(t')+\int_\RR w^2\phi_N(t')\right)-\frac{2}{c_N}\int_\RR wR_N(t'),\\
\end{aligned} 
\end{equation}
that is to: $\forall\;t'\geq t\geq T,$

\begin{multline}
\mathcal{H}(t)\le \sum_{i=1}^N\frac{1}{c_i^2}\left\{(\partial_xw)^2+c_iw^2(t')-pR_i^{p-1}(t')w^2(t')\right\}\phi_i(t') \\
+ C\|w(t')\|_{L^\infty}\|w(t')\|^2_{L^2}+C\tilde{\alpha}\|\epsilon(t)\|^2_{L^2}+
Ce^{-\frac{\nu^{\frac{3}{2}}}{4}t},
\end{multline}
where $C$ does not depend on $\tilde{\alpha}$.

\begin{Rq}\label{rk:kappa}
Note that the monotonicity property \eqref{ineqM} can indeed be used in the preceding estimates since $\kappa$ verifies $\kappa\left(\frac{1}{c_i}+\frac{1}{c_{i+1}}\right)<\frac{1}{2}$ and $\frac{\kappa}{c_N}<\frac{1}{2}$.
\end{Rq}

\noindent Assumption (\ref{lim0}) tells us exactly that $\|w(t')\|_{H^1}\underset{t'\to+\infty}{\longrightarrow} 0,$ thus we obtain \begin{equation}\label{ineq_H}
\forall\;t\geq T,\quad \mathcal{H}(t)\leq Ce^{-\frac{\nu^{\frac{3}{2}}}{4}t}+C\alpha\|\epsilon(t)\|_{L^2}^2.
\end{equation}

\begin{Rq}
Notice that it is important to consider $w$ instead of $\epsilon$ in estimates \eqref{ineqM} and \eqref{ineqE} to obtain \eqref{ineq_H}, thus to improve the a priori control of $\mathcal{H}$ by $\mathrm{O}\left(\|\epsilon\|_{H^1}^2\right)$.
\end{Rq}

\noindent Then, by (\ref{eg3}) and by the following estimate
\begin{equation}
\sum_{i=1}^N\left(\int_\RR\epsilon(t)\tilde{R}_i(t)\right)^2\leq Ce^{-\frac{\nu^{\frac{3}{2}}}{4}t}\|\epsilon(t)\|_{L^2}+C\tilde{\alpha}\|\epsilon(t)\|^2_{L^2}
\end{equation}
(see Martel \cite[Step 3]{martel}) for a proof), there exists $C_0>0$ such that for all $t \geq T$,
\begin{align*}
\|\epsilon(t)\|^2_{H^1}&\leq\displaystyle Ce^{-\frac{\nu^{\frac{3}{2}}}{4}t}+C\tilde{\alpha}\|\epsilon(t)\|^2_{H^1}+C\sum_{i=1}^N\left(\int_\RR\epsilon(t)\tilde{R}_i(t)\right)^2 \leq C_0e^{-\frac{\nu^{\frac{3}{2}}}{4}t}+C_0\tilde{\alpha}\|\epsilon(t)\|^2_{H^1}.
\end{align*}

\noindent Now, due to the independence of $C_0$ with respect to $\tilde{\alpha}$, even if it means taking a smaller $\tilde{\alpha}$ so that $ C_0\tilde{\alpha}<1$, we infer 
\begin{equation}\label{estepsilon}
\forall\;t\geq T,\qquad \|\epsilon(t)\|_{H^1}^2\leq Ce^{-\frac{\nu^{\frac{3}{2}}}{4}t}.
\end{equation}
By (\ref{est2}), this implies that 
$$\forall\;t\geq T,\qquad |x_i'(t)+y_i'(t)-c_i|\leq Ce^{-\frac{\nu^{\frac{3}{2}}}{8}t},$$ which leads to the existence of $y_i\in\RR$ such that $$x_i(t)+y_i(t)-c_it\underset{t\to+\infty}{\longrightarrow}y_i$$
and \begin{equation}\label{ineq:param}
|x_i(t)+y_i(t)-c_it-y_i|\le Ce^{-\frac{\nu^{\frac{3}{2}}}{8}t}.
\end{equation}

\noindent Hence, using (\ref{estepsilon}), the triangular inequality, the following estimate
$$\|Q_{c_i}(\cdot-x_i(t)-y_i(t))-Q_{c_i}(\cdot-c_it-y_i)\|_{H^1}\le C|x_i(t)+y_i(t)-c_it-y_i|$$
(which is a consequence of Lemma \ref{lem:Q} below), and \eqref{ineq:param}, we have \begin{equation}\label{expconvsigma}
\left\|u(t)-\sum_{i=1}^NQ_{c_i}(\cdot-c_it-y_i)\right\|_{H^1}\leq Ce^{-\frac{\nu^{\frac{3}{2}}}{8}t}.
\end{equation}

\begin{Lem}\label{lem:Q}
For all $i=1,\dots,N$, for all $r\ge 0$, and for all $s\in\NN^*$, 
$$\left\|\partial_{x}^sQ_{c_i}(\cdot-r)-\partial_{x}^sQ_{c_i}\right\|_{L^2}^2\le \left(\|\partial_x^{s+1}Q_{c_i}\|_{L^2}^2+(r+2z_s)\|\partial_x^{s+1}Q_{c_i}\|_{L^\infty}^2\right)r^2,$$
where $z_s:=\max\left(\left(\partial_x^{s+1}Q_{c_i}\right)^{-1}(\{0\})\cap\RR^*_+\right)$.
\end{Lem}

\begin{proof}[Proof of Lemmma \ref{lem:Q}]
By the mean value theorem, we have:
$$\left\|\partial_{x}^sQ_{c_i}(\cdot-r)-\partial_{x}^sQ_{c_i}\right\|_{L^2}^2=\int_\RR\left(\partial_{x}^{s+1}Q_{c_i}(\xi_x)\right)^2r^2\;dx,$$ where $x-r\le\xi_x\le x$ for all $x\in\RR$. \\
Now, split the preceding integral into three regions: $x\le -z_s$, $-z_s\le x\le z_s+r $, and $x\ge z_s+r$. In the first and third regions, use the monotonicity of $\left(\partial_{x}^{s+1}Q_{c_i}\right)^2$. We have:
$$\begin{dcases}
\left(\partial_{x}^{s+1}Q_{c_i}(\xi_x)\right)^2\le \left(\partial_{x}^{s+1}Q_{c_i}(x)\right)^2 & \text{if }
x\le -z_s\\
\left(\partial_{x}^{s+1}Q_{c_i}(\xi_x)\right)^2\le \left(\partial_{x}^{s+1}Q_{c_i}(x-r)\right)^2 & \text{if }
x\ge z_s+r.
\end{dcases}$$
If $-z_s\le x\le z_s+r$, we have $$\left(\partial_{x}^{s+1}Q_{c_i}(\xi_x)\right)^2\le \|\partial_{x}^{s+1}Q_{c_i}\|_{L^\infty}^2.$$
Thus, we obtain:
$$\left\|\partial_{x}^sQ_{c_i}(\cdot-r)-\partial_{x}^sQ_{c_i}\right\|_{L^2}^2\le \left(\|\partial_x^{s+1}Q_{c_i}\|_{L^2(|x|\ge z_s)}^2+(r+2z_s)\|\partial_x^{s+1}Q_{c_i}\|_{L^\infty}^2\right)r^2,$$
which puts an end to the proof.
\end{proof}

At the stage of \eqref{expconvsigma}, it suffices to see a posteriori that $\delta_0\leq c_1$ in order to obtain exactly (\ref{limexp}). Let us justify it briefly. \\

\noindent For all $i\in\{1,\ldots,N\}$, set $f_i(t,x):=Q_{c_i}(x-x_i(t))-Q_{c_i}(x-c_it-y_i)$. \\
From (\ref{lim0}) and (\ref{expconvsigma}), we deduce
\begin{equation}\label{lim0sum}
\left\|\sum_{i=1}^Nf_i(t)\right\|_{H^1}\to 0, \qquad \text{as } t\to +\infty.
\end{equation}
Define $th_p(x):=\mathrm{tanh}\big(\frac{p-1}{2}x\big)$. A direct computation yields $$\forall\;x\in\RR,\quad Q'(x)=-th_p(x)Q(x),$$
from which we have for all $i\in\{1,\ldots,N\}$
\begin{align*}
\partial_xf_i(t,x)=&-\sqrt{c_i}th_p\big(\sqrt{c_i}(x-c_it-y_i)\big)f_i(t,x)\\
&+\sqrt{c_i}\Big(th_p\big(\sqrt{c_i}(x-x_i(t))\big)-th_p\big(\sqrt{c_i}(x-c_it-y_i)\big)\Big)Q_{c_i}(x-x_i(t)).
\end{align*}

\noindent By (\ref{lim0sum}), we have in particular $$\int_{\RR}\left(\sum_{i=1}^N\partial_xf_i(t,x)\right)^2\;dx\to 0, \qquad \text{as } t\to +\infty.$$ 

\noindent Now, it follows from Lebesgue's dominated convergence theorem that 
$$\left\|\sum_{i=1}^N\sqrt{c_i}f_i(t)\right\|_{L^2}\to 0, \qquad \text{as } t\to +\infty.$$
Combining this result with (\ref{lim0sum}) and due to the fact that the speeds $c_i$ are distinct two by two, we obtain successively for $k$ describing the integers from $N-1$ to $1$:
$$\left\|\sum_{i=1}^{k}\big(\sqrt{c_i}-\sqrt{c_{k+1}}\big)f_i(t)\right\|_{L^2}\to 0, \qquad \text{as } t\to +\infty.$$

\noindent Hence $\|f_1(t)\|_{L^2}\underset{t\to+\infty}{\longrightarrow} 0$ and even $\|f_1(t)\|_{H^1}\underset{t\to+\infty}{\longrightarrow} 0$ judging by the expression of $\partial_xf_1$.  This implies $x_1(t)-c_1t-y_1\underset{t\to+\infty}{\longrightarrow} 0$. Now it is clear that condition $x_1'(t)\geq\delta_0$ forces to have $c_1\geq \delta_0$.

\bibliographystyle{plain}
\bibliography{references}

\end{document}